\documentclass{amsart}
\usepackage{amsfonts,amssymb,latexsym}
\usepackage{amsmath,amsthm}
\usepackage{eucal}
\usepackage[latin1]{inputenc}

\newtheorem{theorem}{Theorem}[section]
\newtheorem{lemma}[theorem]{Lemma}
\newtheorem{proposition}[theorem]{Proposition}

\theoremstyle{definition}

\theoremstyle{remark}
\newtheorem{remark}[theorem]{Remark}
\def\R{{\mathbb R}}

\newcommand{\Z}{\mathbb{Z}}

\numberwithin{equation}{section}

\title[Higher order Benjamin-Ono equation]{Global well-posedness and limit behavior for a higher-order Benjamin-Ono
equation}
\author[L. Molinet and D. Pilod]{Luc Molinet$^{\dagger}$ and Didier Pilod$^{\ddagger}$}
\thanks{$^{\ddagger}$ Partially supported by CNPq/Brazil, grant 200001/2011-6}
\date{}
\begin{document}
\maketitle

\vspace{-0.5cm}

{\scriptsize \centerline{$^{\dagger}$ LMPT, Université François Rabelais
Tours, Fédération Denis Poisson-CNRS,}
             \centerline{ Parc Grandmont, 37200 Tours, France.}
             \centerline{email: Luc.Molinet@lmpt.univ-tours.fr}

\vspace{0.1cm}
             \centerline{$^{\ddagger}$ UFRJ, Instituto de Matemática,
             Universidade Federal do Rio de Janeiro,}
              \centerline{Caixa Postal 68530, CEP: 21945-970, Rio
              de Janeiro, RJ, Brazil.}
              \centerline{email: didier@im.ufrj.br}}

\vspace{0.5cm}

\begin{abstract} 
In this paper, we prove that the Cauchy problem associated to the following higher-order Benjamin-Ono equation
\begin{equation} \label{ahoBO}
\partial_tv-b\mathcal{H}\partial^2_xv-
           a\epsilon \partial_x^3v=cv\partial_xv-d\epsilon
           \partial_x(v\mathcal{H}\partial_xv+\mathcal{H}(v\partial_xv)),
\end{equation}
is globally well-posed in the energy space $H^1(\mathbb R)$. Moreover, we study the limit behavior when the small positive parameter $\epsilon$ tends to zero and show that, under a condition on the coefficients $a$, $b$, $c$ and $d$, the solution $v_{\epsilon}$ to \eqref{ahoBO} converges to the corresponding  solution of the Benjamin-Ono equation.
\end{abstract}

\section{Introduction}
Considered here is the following higher-order Benjamin-Ono equation
\begin{equation} \label{hoBO}
\partial_tv-b\mathcal{H}\partial^2_xv-
           a\epsilon \partial_x^3v=cv\partial_xv-d\epsilon
           \partial_x(v\mathcal{H}\partial_xv+\mathcal{H}(v\partial_xv)),
\end{equation}
where $x$, $t \in \mathbb R$, $v$ is a real-valued function, $a$,
$b, \ c$ and $d$ are positive constants, $\epsilon>0$ is a small
positive parameter and $\mathcal{H}$ is the Hilbert transform,
defined on the line by
\begin{equation} \label{hilbert}
\mathcal{H}f(x)=\text{p.v.}\,\frac1\pi\int_{\mathbb
R}\frac{f(y)}{x-y}dy.
\end{equation}

The equation above corresponds to a second order approximation of
the unidirectional evolution of weakly nonlinear dispersive internal
long waves at the interface of a two-layer system of fluids, the
lower one being infinitely deep. It was derived by Craig, Guyenne
and Kalisch (see equation (5.38) in \cite{CGK}), using a Hamiltonian
perturbation theory. Here, $v$ represents the dislocation of the
interface around its position of equilibrium, the coefficients $a$,
$b$, $c$ and $d$ are respectively given by
\begin{equation} \label{ab}
a=\frac{h_1^2}2\big(\frac{\rho^2}{\rho_1^2}-\frac13
\big)\sqrt{\frac{gh_1(\rho-\rho_1)}{\rho_1}}, \quad \quad
b=\frac{\rho
h_1^2}{2\rho_1^2}\sqrt{\frac{g\rho_1(\rho-\rho_1)}{h_1}},
\end{equation}
\begin{equation} \label{cd}
c=\frac{3\sqrt{2}}{4\rho_1}\sqrt[4]{\frac{g\rho_1(\rho-\rho_1)}{h_1}}
\quad \text{and} \quad d=\frac{\sqrt{2}\rho
h_1}{2\rho_1^2}\sqrt[4]{\frac{g\rho_1(\rho-\rho_1)}{h_1}},
\end{equation}
where $h_1$ represents the depth of the upper layer when the fluid
is at rest, $\rho_1$ is the density of the upper fluid and $\rho$ is
the density of the lower fluid. Moreover, the system is assumed to
be in a stable configuration, which is to say  that $\rho>\rho_1$,
so that the coefficients $a$, $b$, $c$ and $d$ are positive.

It is worth noting that the equation obtained at the first order approximation  of the above physical model is the well-known Benjamin-Ono equation
\begin{equation} \label{BO} 
\partial_tv-b\mathcal{H}\partial^2_xv=cv\partial_xv,
\end{equation}
and therefore equation \eqref{hoBO} can be seen as an higher-order perturbation of equation \eqref{BO}. Moreover,
the quantities
\begin{equation} \label{M}
M(v)=\int_{\mathbb R}v^2dx
\end{equation}
and
\begin{equation} \label{H}
H(v)=\int_{\mathbb R}\left(a\epsilon(\partial_xv)^2-bv\mathcal{H}\partial_xv
-\frac{c}{3}v^3+d\epsilon v^2\mathcal{H}\partial_xv\right)dx
\end{equation}
are conserved by the flow associated to \eqref{hoBO}. 

The initial value problem (IVP) associated to the Benjamin-Ono equation on the line has been extensively studied in the recent years and has been proved to be globally well-posed in $L^2(\mathbb R)$ by Ionescu and Kenig in \cite{IK} (see \cite{MP} for another proof and \cite{abfs, BP, Io, KK, KT,  Po1, Ta} for former results).  The IVP associated to \eqref{hoBO} presents the same mathematical difficulties as for the Benjamin-Ono equation. Indeed, it has been shown in \cite{Pi} that the flow map data-solution cannot be $C^2$ in any $L^2$-based Sobolev space $H^s(\mathbb R)$, $s \in \mathbb R$, by using the same counter-example as for the Benjamin-Ono equation in \cite{MST}. On the other hand,  the Cauchy problem associated to \eqref{hoBO} was proved in \cite{LPP} to be locally well-posed in $H^s(\mathbb R)$, for $s \ge 2$ (and also in weighted Sobolev spaces $H^k(\mathbb R) \cap L^2(\mathbb R; x^2dx)$, for $k \in \mathbb Z_+$, $k \ge 2$). However, there are no conserved quantities at the $H^2$ level and thus it is not known wether these local solutions extend globally in time or not. Therefore, as commented in \cite{LPP}, the question of the local well-posedness in $H^1(\mathbb R)$, which would directly imply global well-posedness by using \eqref{M} and \eqref{H}, arises naturally. 

The first aim of this paper is to give a positive answer to this issue. The result states as follows. 
\begin{theorem} \label{theo1} 
Fix $\epsilon>0$ and let $s \ge 1$ be given. Then, for all $v_0 \in H^s(\mathbb R)$ and all $T>0$, there exists a unique solution 
$v$ to equation \eqref{hoBO} in the space
\begin{equation} \label{theo1.1} 
 C([0,T];H^s(\mathbb R))\cap L^4_TW^{s,4}_x\cap L^2_xL^{\infty}_T\cap  X^{s-2\theta,\theta}_{\epsilon,T}, \quad \text{for all} \ 0 \le \theta \le 1.
\end{equation}
satisfying 
\begin{equation}
v(\cdot,0)=v_0
\end{equation} 
and 
\begin{equation} \label{theo1.2} 
w=\partial_xP_{+hi}\big(e^{-iF[v]}\big) \in X^{s,\frac12,1}_{\epsilon,T},
\end{equation}
where $F[v]$ is a spatial primitive of $v$ defined in Section \ref{gt}. 

Moreover, $v \in C_b(\mathbb R;H^1(\mathbb R))$ and the flow map data-solution $S_{\epsilon}:v_0 \mapsto v$ is continuous from $H^s(\mathbb R)$ into $C([0,T];H^s(\mathbb R))$.
\end{theorem}
Note that above $H^s(\mathbb R)$ denotes the space of all real-valued functions with the usual norm, while $X^{s,b}_{\epsilon,T}$ and $X^{s,b,q}_{\epsilon,T}$ are Bourgain spaces defined in Subsection \ref{spaces}. 

Since it follows from the result of ill-posedness in \cite{Pi} that the Cauchy problem associated to \eqref{hoBO} cannot be solved by using a fixed point theorem on the integral equation, we use a compactness argument based on the smooth solutions obtained in \cite{LPP}. To derive \textit{a priori} estimates at the $H^1$ level, we introduce a gauge transformation which weakens the high-low frequency interactions in the nonlinearity of \eqref{hoBO}, as it was done by Tao in \cite{Ta} for the Benjamin-Ono equation. Note that the same kind of gauge transformation was already introduced in \cite{LPP} to obtain the solutions in $H^2(\mathbb R)$.  However, to lower the regularity till $H^1(\mathbb R)$, we will need to combine this transformation with the use of Bourgain's spaces (as it was already done in \cite{BP, IK, MP} for BO).  More precisely, we need to work in a Besov version of Bourgain's spaces (introcuded in \cite{Tat} in the context of waves maps). Indeed, on one hand we have to work in Bourgain' spaces of conormal regularity $ 1/2 $ to establish the main bilinear estimate (see Proposition \ref{bilin} below). On the other hand, to control some remaining terms appearing in the transformation, we need the full Kato smoothing effect for functions that are localized in space frequencies (see Proposition \ref{bilinb}). The rest of the proof follows closely the one in \cite{MP} for the Benjamin-Ono equation (see also \cite{Mo}).

In the second part of this article, we investigate the limit behavior of the solutions $v_{\epsilon}$ to \eqref{hoBO}, obtained in Theorem \ref{theo1}, as $\epsilon$ tends to zero.  First, it is interesting to observe that a direct argument based on compactness methods (see for example \cite{Mo2} in the case of the Benjamin-Ono-Burgers equation) does not seem to work. Indeed, the leading terms in the energy $H$, which is to say $a\epsilon(\partial_xv)^2$ and $bv\mathcal{H}v$, have opposite signs, so that \eqref{M} and \eqref{H} do not provide \textit{a priori} bounds, uniformly in $\epsilon$, on $\epsilon\|v_{\epsilon}\|_{H^1}^2+\|v_{\epsilon}\|_{H^{\frac12}}^2$.  Therefore, the problem of studying the limit of $v_{\epsilon}$, as $\epsilon$ goes to zero, turns out to be far from trivial.

Nevertheless, we are able to prove the convergence of solutions of \eqref{hoBO} toward a solution of the Benjamin-Ono equation in the special case where the ratio of the densities is equal to $\sqrt{3}$. 
 \begin{theorem} \label{theo2} 
Assume that  $ \frac{3ac}{4d}=b \quad
\Leftrightarrow \quad \rho^2=3\rho_1^2$. Let  $v_0 \in H^1(\mathbb R)$ and for any $ \varepsilon >0 $ denote by  $ S_\varepsilon(t) v_0 \in C(\R; H^1(\R)) $   the solution to  \eqref{hoBO} emanating from $ v_0$. Then for any $ T>0 $ it holds 
\begin{equation}
\|S_\varepsilon (t) v_0 - S(t) v_0 \|_{L^\infty(0,T; H^1(\R))} \longrightarrow 0 \mbox{ as } \varepsilon \to 0  
\end{equation}
where $ S(t) v_0 $ is the solution to the Benjamin-Ono equation emanating from $ v_0$.
\end{theorem}

In the case where $\frac{\rho}{\rho_1}=\sqrt{3}$, the spatial primitive chosen to perform the gauge transformation for equation \eqref{hoBO} corresponds to the one chosen for the Benjamin-Ono equation. Then, we can show that the Cauchy problem associated to \eqref{hoBO} is uniformly in $\epsilon$ well-posed in $H^1(\mathbb R)$, which will in a classical way (see for example \cite{GW}) lead to Theorem \ref{theo2}. The main difficulty here arises from the fact that the dispersive linear terms $\epsilon \partial_x^3$ and $\mathcal{H}\partial_x^2$ compete together as in the Benjamin equation (see the introduction in \cite{ABR}). Therefore, we are only allowed to use the dispersive smoothing effects associated to \eqref{hoBO} in some well behaved regions in spatial frequency and we need to refine the bilinear estimates obtained in the proof of Theorem \ref{theo1} in the other regions. 

It would be interesting to derive a class of higher-order equation for internal long waves from the first order one derived by Bona, Lannes and Saut in \cite{BLS}.  Among those equations, which would be formally equivalent to \eqref{hoBO},  one might find some with better behaved linear parts, which would avoid to deal with those technical difficulties. 

Finally, we observe that the techniques introduced here would likely lead to similar results for the following intermediate long wave equation 
\begin{equation}\label{ho.ILW}
\partial_tu-b\mathcal{F}_h\partial^2_xu+(a_1\mathcal{F}_h^2+a_2)\epsilon \partial_x^3u
=cu\partial_xu-d\epsilon \partial_x(u\mathcal{F}_h\partial_xu+\mathcal{F}_h(u\partial_xu))
\end{equation}
where $\mathcal{F}_h$ is the Fourier multiplier $-i\coth(h\xi)$, $u$
is a real-valued solution, and $a_1, \ a_2, \ b, \ c, \ d$ and $h$
are positive constants, and which was also derived in \cite{CGK}. Note that the same ill-posedness results as for equation \eqref{hoBO} also hold for this equation (see \cite{Pi}). 

The paper is organized as follows: in the next section, we introduce the notations, define the functions spaces and recall some classical estimates. Sections 3 and 4 are devoted the key nonlinear estimates, which are used in Section 5 to prove Theorem \ref{theo1}. Finally, in Section 6, we prove Theorem \ref{theo2}.

\section{Notations, function spaces and preliminary estimates} \label{not}

\subsection{Notation}
For any positive numbers $a$ and $b$, the notation $a \lesssim b$
means that there exists a positive constant $c$ such that $a \le c
b$. We also denote $a \sim b$ when $a \lesssim b$ and $b \lesssim
a$. Moreover, if $\alpha \in \mathbb R$, $\alpha_+$, respectively
$\alpha_-$, will denote a number slightly greater, respectively
lesser, than $\alpha$.

For $u=u(x,t) \in \mathcal{S}(\mathbb R^2)$,
$\mathcal{F}u=\widehat{u}$ will denote its space-time Fourier
transform, whereas $\mathcal{F}_xu=(u)^{\wedge_x}$, respectively
$\mathcal{F}_tu=(u)^{\wedge_t}$, will denote its Fourier transform
in space, respectively in time. For $s \in \mathbb R$, we define the
Bessel and Riesz potentials of order $-s$, $J^s_x$ and $D_x^s$, by
\begin{displaymath}
J^s_xu=\mathcal{F}^{-1}_x\big((1+|\xi|^2)^{\frac{s}{2}}
\mathcal{F}_xu\big) \quad \text{and} \quad
D^s_xu=\mathcal{F}^{-1}_x\big(|\xi|^s \mathcal{F}_xu\big).
\end{displaymath}

Throughout the paper, we fix a smooth cutoff function $\eta$ such that
\begin{displaymath}
\eta \in C_0^{\infty}(\mathbb R), \quad 0 \le \eta \le 1, \quad
\eta_{|_{[-1,1]}}=1 \quad \mbox{and} \quad  \mbox{supp}(\eta)
\subset [-2,2].
\end{displaymath}
Then if $A$ is a positive number, $P_{\lesssim A}$ denote the Fourier multiplier whose symbol is given by $\eta(\frac{\cdot}{cA})$ and $P_{\gtrsim A}$ is defined by $P_{\gtrsim A}=1-P_{\lesssim A}$.
For $l \in \mathbb Z_{+}$, we define
\begin{displaymath}
\phi(\xi):=\eta(\xi)-\eta(2\xi), \quad
\phi_{2^l}(\xi):=\phi(2^{-l}\xi), \end{displaymath} and
\begin{displaymath}
\psi_{2^{l}}(\xi,\tau)=\phi_{2^{l}}(\tau-b|\xi|\xi+a\epsilon\xi^3).
\end{displaymath}
By convention, we also denote
\begin{displaymath}
\phi_0(\xi):=\eta(2\xi), \quad \text{and} \quad
\psi_{0}(\xi):=(\xi,\tau)=\phi_{0}(2(\tau-b|\xi|\xi+a\epsilon\xi^3)),
\end{displaymath}
Any summations over capitalized variables such as $N, \, L$, $K$ or
$M$ are presumed to be dyadic with $N, \, L$, $K$ or $M \ge 0$,
\textit{i.e.}, these variables range over numbers of the form $\{2^n
: n \in \mathbb Z_{+}\} \cup \{0\}$. Then, we have that
\begin{displaymath}
\sum_{N}\phi_N(\xi)=1, \quad \mbox{supp} \, (\phi_N) \subset
\{\frac{N}{2}\le |\xi| \le 2N\}, \ N \ge 1, \quad \text{and} \quad
\mbox{supp} \, (\phi_0) \subset \{|\xi| \le 1\}.
\end{displaymath}
Let us define the Littlewood-Paley multipliers by
\begin{displaymath}
P_Nu=\mathcal{F}^{-1}_x\big(\phi_N\mathcal{F}_xu\big), \quad
Q_Lu=\mathcal{F}^{-1}\big(\psi_L\mathcal{F}u\big),
\end{displaymath}
and $P_{\ge N}:=\sum_{K \ge N} P_{K}$. Moreover, we also define the
operators $P_{hi}$, $P_{HI}$, $P_{lo}$ and $P_{LO}$ by
\begin{displaymath}
P_{hi}=\sum_{N\ge 2 } P_N, \quad  P_{HI}=\sum_{N \ge 2^4}P_N, \quad
P_{lo}= 1-P_{hi}, \quad \text{and} \quad P_{LO}= 1-P_{HI}.
\end{displaymath}

Let $P_+$ and $P_-$ denote the projection on respectively the
positive and the negative Fourier frequencies. Then
\begin{displaymath}
P_{\pm}u=\mathcal{F}^{-1}_x\big(\chi_{\mathbb
R_{\pm}}\mathcal{F}_xu\big),
\end{displaymath}
and we also denote $P_{\pm hi}=P_{\pm}P_{hi}$, $P_{\pm
HI}=P_{\pm}P_{HI}$, $P_{\pm lo}=P_{\pm}P_{lo}$, $P_{\pm
LO}=P_{\pm}P_{LO}$ and $P_{\pm N}=P_{\pm}P_{N}$. Observe that $P_{hi}$, $P_{HI}$, $P_{lo}$,
$P_{LO}$, $P_N$ and $P_{\pm N}$ are bounded (uniformly in $N$) operators on $L^p(\mathbb R)$ for $1\le
p\le\infty$, while $P_{\pm}$ are only bounded on $L^p(\mathbb R)$
for $1 < p < \infty$. We also note that
\begin{displaymath}
\mathcal{H}=-iP_++iP_-.
\end{displaymath}

Finally, we denote by
$V_{\epsilon}(t)=e^{t(b\mathcal{H}\partial_x^2+a\epsilon
\partial_x^3)}$ the free group associated with the linearized part
of equation \eqref{hoBO}, which is to say,
\begin{equation} \label{V}
\mathcal{F}_x\big(V_{\epsilon}(t)f
\big)(\xi)=e^{it(b|\xi|\xi-a\epsilon\xi^3)}\mathcal{F}_xf(\xi).
\end{equation}

\subsection{Function spaces} \label{spaces}
For $1 \le p \le \infty$, $L^p(\mathbb R)$ is the usual Lebesgue
space with the norm $\|\cdot\|_{L^p}$, and for $s \in \mathbb R$ ,
the real-valued Sobolev spaces $H^s(\mathbb R)$ and $W^{s,p}(\mathbb
R)$ denote the spaces of all real-valued functions with the usual
norms
\begin{displaymath}
\|\phi\|_{H^s}=\|J^s_x\phi\|_{L^2} \quad \text{and} \quad
\|\phi\|_{W^{s,p}}=\|J^s_x \phi\|_{L^p}.
\end{displaymath}
If $f=f(x,t)$ is a function defined for $x \in
\mathbb R$ and $t$ in the time interval $[0,T]$, with $T>0$, if $B$
is one of the spaces defined above, $1 \le p \le \infty$ and $1 \le q \le \infty$, we will
define the mixed space-time spaces $L^p_TB_x$, 
$L^p_tB_x$, $L^q_xL^p_T$ by the norms
\begin{displaymath}
\|f\|_{L^p_TB_x} =\Big(
\int_0^T\|f(\cdot,t)\|_{B}^pdt\Big)^{\frac1p} \quad 
\|f\|_{L^p_tB_x} =\Big( \int_{\mathbb R}\|f(\cdot,t)\|_{B}^pdt\Big)^{\frac1p}, 
\end{displaymath}
and 
\begin{displaymath} 
\|f\|_{L^q_xL^p_T}= \left(\int_{\mathbb R}\Big( \int_0^T|f(x,t)|^pdt\Big)^{\frac{q}{p}}dx\right)^{\frac1q}.
\end{displaymath}
Moreover, if $s \in \mathbb R$, $1 \le q \le \infty$ and $X$ denotes one of the mixed space-time spaces defined above, we define its dyadic version $\mathcal{B}^{s,q}(X)$ as 
\begin{displaymath} 
\|f\|_{\mathcal{B}^{s,q}(X)}=\left( \sum_{N}\langle N \rangle^{sq}\|P_Nf\|_{X}^q\right)^{\frac1q}.
\end{displaymath}
In the special case $ (s,q)=(0,2)$, the space $ \mathcal{B}^{s,q}(X)$ will be simply denoted by $ \widetilde{X} $. 

For $s$, $b \in \mathbb R$, we introduce the Bourgain spaces
$X^{s,b}_{\epsilon}$ related to the linear part of \eqref{hoBO} as
the completion of the Schwartz space $\mathcal{S}(\mathbb R^2)$
under the norm
\begin{equation} \label{Bourgain}
\|v\|_{X^{s,b}_{\epsilon}} := \left(
\int_{\mathbb{R}^2}\langle\tau-b|\xi|\xi+a\epsilon
\xi^3\rangle^{2b}\langle \xi\rangle^{2s}|\widehat{v}(\xi, \tau)|^2
d\xi d\tau \right)^{\frac12},
\end{equation}
where $\langle x\rangle:=1+|x|$. We will also use a dyadic version
of those spaces introduced in \cite{Tat} in the context of wave maps. 
For $s$, $b \in \mathbb R$, $1 \le q \le
\infty$, $X^{s,b,q}_{\epsilon}$ will denote the completion of the
Schwartz space $\mathcal{S}(\mathbb R^2)$ under the norm
\begin{equation} \label{Bourgain2}
\|v\|_{X^{s,b,q}_{\epsilon}} := \left( \sum_N\Big(\sum_L\langle N
\rangle^{sq}\langle
L\rangle^{bq}\|P_NQ_Lv\|_{L^2_{x,t}}^q\Big)^{\frac2q}
\right)^{\frac12}.
\end{equation}
Moreover, we define a localized (in time) version of these spaces.
Let $T>0$ be a positive time and $Y=X^{s,b}_{\epsilon}$ or
$Y=X^{s,b,q}_{\epsilon}$. Then, if $v: \mathbb R \times
[0,T]\rightarrow \mathbb C$, we have that
\begin{displaymath}
\|v\|_{Y_{T}}:=\inf \{\|\tilde{v}\|_{Y} \ | \ \tilde{v}: \mathbb R
\times \mathbb R \rightarrow \mathbb C, \ \tilde{v}|_{\mathbb R
\times [0,T]} = v\}.
\end{displaymath}
When $\epsilon=1$, we will denote $X^{s,b}=X^{s,b}_{1}$,
$X^{s,b}_{T}=X^{s,b}_{1,T}$, $X^{s,b,q}=X^{s,b,q}_{1}$ and
$X^{s,b,q}_{T}=X^{s,b,q}_{1,T}$.

Finally we list some useful properties of the  Bourgain spaces
defined above.
\begin{proposition} \label{propBourgain}
Fix $\delta>0$, $s \in \mathbb R$ and $\epsilon>0$. Then it holds
that
\begin{equation} \label{propBourgain1}
\|v\|_{X^{s,\frac12}_{\epsilon}} \lesssim
\|v\|_{X^{s,\frac12,1}_{\epsilon}} \lesssim
\|v\|_{X^{s,\frac12+\delta}_{\epsilon}},
\end{equation}
\begin{equation} \label{propBourgain2}
\|v\|_{L^{\infty}_tH^s_x} \lesssim
\|\widehat{J^s_xv}\|_{L^2_{\xi}L^1_{\tau}} \lesssim
\|v\|_{X^{s,\frac12,1}_{\epsilon}},
\end{equation}
and
\begin{equation} \label{propBourgain1b}
\|f\|_{X^{s,-\frac12+\delta}_{\epsilon}} \lesssim
\|f\|_{L^{1+\delta'}_tH^s_x},
\end{equation}
for $\delta'>0$ satisfying $1+\delta'=\frac1{1-\delta}$. In other
words, the injections
\begin{displaymath} X^{s,\frac12+\delta}_{\epsilon}
\hookrightarrow X^{s,\frac12,1}_{\epsilon} \hookrightarrow
X^{s,\frac12}_{\epsilon}, \quad  X^{s,\frac12,1}_{\epsilon}
\hookrightarrow L^{\infty}_tH^s_x,
\end{displaymath}
and
\begin{displaymath}
L^{1+\delta'}_tH^s_x \hookrightarrow
X^{s,-\frac12+\delta}_{\epsilon}
\end{displaymath}
are continuous.
\end{proposition}

\subsection{Linear estimates}
First, we recall some linear estimates in Bourgain's spaces which
will be needed later (see for instance \cite{Tat}).

\begin{lemma}[Homogeneous linear estimate] \label{prop1.1}
Let $s \in \mathbb R$ and $\epsilon>0$. Then
\begin{equation} \label{prop1.1.2}
\|\eta(t)V_{\epsilon}(t)\phi\|_{X^{s,\frac12,1}_{\epsilon}}
\lesssim\|\phi\|_{H^s}.
\end{equation}
\end{lemma}
\begin{lemma}[Non-homogeneous linear estimate] \label{prop1.2}
Let $s \in \mathbb R$ and $\epsilon>0$. Then, it holds that
\begin{equation} \label{prop1.2.1}
\big\|\eta(t)\int_0^tV_{\epsilon}(t-t')g(t')dt'\big\|_{X^{s,\frac12,1}_{\epsilon}}
\lesssim  \|g\|_{X^{s,-\frac12,1}_{\epsilon}}.
\end{equation}
\end{lemma}


Next, we derive local and global smoothing effects associated to the group $\{V_{\epsilon}(t)\}$, for the KdV scaling, in the context of Bourgain's spaces. We begin with the Strichartz estimates.
\begin{lemma} \label{strichartz}
For all $0 <\epsilon <1$, $T>0$ and $0 \le \theta \le 1$, we have
that
\begin{equation} \label{strichartz7}
\|v\|_{L^{p_{\theta}}_{x,t}} \lesssim \|v\|_{\widetilde{L^{p_{\theta}}_{x,t}}} \lesssim
\epsilon^{-\frac{\theta}8}\|v\|_{X^{0,\frac{\theta}{2}+}_{\epsilon}},
\end{equation}
and
\begin{equation} \label{strichartz8}
\|v\|_{L^{p_{\theta}}_{x,T}} \lesssim
\epsilon^{-\frac{\theta}8}\|v\|_{X^{0,\frac{\theta}{2}+}_{\epsilon,T}},
\end{equation}
where $\frac{1}{p_{\theta}}=\frac{\theta}8+\frac{1-\theta}2$.
\end{lemma}

\begin{proof}
First, we observe, arguing as in Lemma 2.1 in \cite{LPP}, that $w$
is a solution to the linear equation
\begin{equation} \label{strichartz1}
\partial_tw-a\epsilon\partial_x^3w\pm ib\partial_x^2w=0,
\end{equation}
if and only if
\begin{equation} \label{strichartz2}
u(x,t)=e^{\pm i\frac{2b^3}{27a^2\epsilon^2}t}e^{\mp i
\frac{b}{3a\epsilon}x}w(x-\frac{b^2t}{3a\epsilon},t)
\end{equation}
is a solution to
\begin{equation} \label{strichartz3}
\partial_tu-a\epsilon\partial_x^3u=0.
\end{equation}

Let us denote by $\{W_{\epsilon}^{\pm}(t)\}$ and
$\{U_{\epsilon}(t)\}$ the groups associated to \eqref{strichartz1}
and \eqref{strichartz3}. Since $U_{\epsilon}(t)=U_1(\epsilon t)$, we
deduce from the classical Strichartz estimate for the KdV equation
(cf. for example \cite{LP}, chapter 4) that
\begin{equation} \label{strichartz4}
\|U_{\epsilon}(t)\phi\|_{L^8_{x,t}} \lesssim
\epsilon^{-\frac18}\|\phi\|_{L^2}.
\end{equation}
Then, it follows gathering \eqref{strichartz1}--\eqref{strichartz4}
with the identity
\begin{equation} \label{strichartz10}
V_{\epsilon}(t)=W^+_{\epsilon}(t)P_++W^-_{\epsilon}(t)P_-,
\end{equation}
that
\begin{equation} \label{strichartz5}
\|V_{\epsilon}(t)\phi\|_{L^8_{x,t}} \lesssim
\epsilon^{-\frac18}\|\phi\|_{L^2}.
\end{equation}

Next, we use Lemma 3.3 in \cite{Gi} to rewrite estimate
\eqref{strichartz5} in the context of Bourgain's spaces. We get that
\begin{equation} \label{strichartz6}
\|v\|_{L^8_{x,t}} \lesssim \epsilon^{-\frac18}
\|v\|_{X^{0,\frac12+}}.
\end{equation}
Therefore, we deduce by using Stein's theorem to interpolate
estimate \eqref{strichartz6} with Plancherel's identity
$\|v\|_{L^2_{x,t}}=\|v\|_{X^{0,0}_{\epsilon}}$, that
\begin{equation} \label{strichartz9}
\|v\|_{L^{p_{\theta}}_{x,t}} \lesssim
\epsilon^{-\frac{\theta}8}\|v\|_{X^{0,\frac{\theta}{2}+}_{\epsilon}}.
\end{equation}

Finally, estimate \eqref{strichartz7} follows directly by applying estimate \eqref{strichartz9} to each dyadic block of $\|v\|_{\widetilde{L^{p_{\theta}}_{x,t}}}$.
\end{proof} 

Next, we turn to the local Kato type smoothing effect.  
\begin{lemma} \label{smoothing} 
Let $0 < \epsilon \le 1$ and $T>0$ and $N \gtrsim \frac1{\epsilon}$. Then, it holds that 
\begin{equation} \label{smoothing1}
\|\partial_xP_Nv\|_{L^{\infty}_xL^2_t} \lesssim
\epsilon^{-\frac12}\|P_Nv\|_{X^{0,\frac12,1}_{\epsilon}},
\end{equation}
and 
\begin{equation} \label{smoothing4}
\|\partial_xv\|_{\widetilde{L^{\infty}_xL^2_T}} \lesssim T^{\frac12}\epsilon^{-{\frac32+}}\|v\|_{X^{0,\frac12,1}_{\epsilon}}.
\end{equation}
\end{lemma}  

\begin{proof} Since $N \gtrsim \frac1{\epsilon}$, we obtain applying estimate (4.3) in Theorem 4.1 of \cite{KPV3} that 
\begin{equation} \label{smoothing2}
\begin{split}
\|\partial_xV_{\epsilon}(t)P_Nv_0\|_{L^{\infty}_xL^2_t} &\lesssim 
\Big(\int_{|\xi| \gtrsim \frac1{\epsilon}}\frac{|\xi|^2}{|2b\xi-3a\epsilon\xi^2|}|\big(P_Nv_0\big)^{\wedge}(\xi)|^2d\xi\Big)^{\frac12} \\ & \lesssim \epsilon^{-\frac12}\|P_{N}v_0\|_{L^2_x}.
\end{split}
\end{equation}
Moreover, by applying the Fourier inverse formula, it follows that 
\begin{displaymath} 
\partial_xP_Nv(x,t)=\int_{\mathbb R}\partial_xV_{\epsilon}(t)\big(V_{\epsilon}(-\cdot)P_Nv \big)^{\wedge_t}(x,\tau)e^{it\tau}d\tau.
\end{displaymath}
Therefore, Minkowski's inequality, estimate \eqref{smoothing2}, Plancherel's identity and the Cauchy-Schwarz inequality imply that
\begin{equation} \label{smoothing2b}
\begin{split}
\|\partial_xP_Nv\|_{L^{\infty}_xL^2_t} & \lesssim  \int_{\mathbb R}\big\|\big(V_{\epsilon}(-\cdot)P_Nv \big)^{\wedge}(\cdot,\tau) \big\|_{L^2_{\xi}}d\tau\\ 
& \lesssim \sum_{L} \langle L\rangle^{\frac12}\big\|\phi_L(\tau)\big(V_{\epsilon}(-\cdot)P_Nv \big)^{\wedge}\big\|_{L^2_{\xi,\tau}},
\end{split}
\end{equation}
which leads to estimate \eqref{smoothing1} since 
$$\big(V_{\epsilon}(-\cdot)P_Nv \big)^{\wedge}(\xi,\tau)=\big(P_Nv \big)^{\wedge}(\xi,\tau+b|\xi|\xi-a\epsilon\xi^3).$$

On the other hand, if $N \lesssim \frac1{\epsilon}$, we deduce from the Sobolev embedding $H^s(\mathbb R) \hookrightarrow L^{\infty}(\mathbb R)$, whenever $s>\frac12$, that
\begin{displaymath} 
\begin{split}
\|\partial_xV_{\epsilon}(t)P_Nv_0\|_{L^{\infty}_xL^2_T}
&\lesssim T^{\frac12} \|\partial_xV_{\epsilon}(t)P_Nv_0\|_{L^{\infty}_{x,T}}\\
&\lesssim T^{\frac12} \|\partial_xV_{\epsilon}(t)P_Nv_0\|_{L^{\infty}_TH^s_x}
\\ & \lesssim T^{\frac12}\epsilon^{-1}(1+\epsilon^{-s})\|P_Nv_0\|_{L^2_x}.
\end{split}
\end{displaymath}
Therefore, we deduce arguing as above that 
\begin{equation} \label{smoothing3}
\|\partial_xP_Nv\|_{L^{\infty}_xL^2_T} \lesssim
T^{\frac12}\epsilon^{-1}(1+\epsilon^{-s})\|P_Nv\|_{X^{0,\frac12,1}_{\epsilon,T}},
\end{equation}
whenever $N \lesssim \frac1{\epsilon}$.

Estimate \eqref{smoothing4} follows gathering estimates \eqref{smoothing1} and \eqref{smoothing3} and by squaring and summing over $N$.
\end{proof}

Finally, we derive the maximal function estimate.
\begin{lemma} \label{maximal}
Let $s>\frac34$, $0 <\epsilon \le 1$, and $T>0$ be such that $0 <\epsilon T \le 1$. Then, we have
that
\begin{equation} \label{maximal1}
\|v\|_{\widetilde{L^{2}_xL^{\infty}_T}} \lesssim
\epsilon^{-s}\|v\|_{X^{s,\frac12,1}_{\epsilon,T}}.
\end{equation}
\end{lemma} 

\begin{proof} 
The $L^2_x$-maximal function for the KdV group $\{U_1(t)\}$ derived in Theorem 2.7 of \cite{KPV2} implies that 
\begin{equation} \label{maximal2} 
\Big(\int_{\mathbb R}\sup_{|t| \le 1}|U_1(t)u_0(x)|^2dx\Big)^{\frac12} \lesssim \|u_0\|_{H^s},
\end{equation}
if $s > \frac34$. Then, a scaling argument and estimate \eqref{maximal2} yield  
\begin{equation} \label{maximal3}
\begin{split}
\|U_{\epsilon}(t)u_0\|_{L^2_xL^{\infty}_T} &= \|U_1(\epsilon t)u_0\|_{L^2_xL^{\infty}_T} \\
&=\Big(\int_{\mathbb R}\sup_{|s| \le \epsilon T}|U_1(s)u_0(x)|^2dx\Big)^{\frac12}  \lesssim 
\|u_0\|_{H^s}^2,
\end{split}
\end{equation}
since $s>\frac34$ and $0<\epsilon T \le 1$.

Thus, if $w$ and $u$ are the solutions associated to \eqref{strichartz1} and \eqref{strichartz3} with respective initial data $w_0$ and $u_0$, it follows from \eqref{strichartz2} and \eqref{maximal3} that 
\begin{equation} \label{maximal4}
\|w\|_{L^2_xL^{\infty}_T}  = \|u\|_{L^2_xL^{\infty}_T}  \lesssim \|u_0\|_{H^s_x} \lesssim \epsilon^{-s} \|w_0\|_{H^s_x}.
\end{equation}
Therefore, we conclude gathering  \eqref{strichartz10} and \eqref{maximal4} that
\begin{equation} \label{maximal5} 
\|V_{\epsilon}(t)v_0\|_{L^2_xL^{\infty}_T} \lesssim \epsilon^{-s}\|v_0\|_{H^s_x},
\end{equation}
whenever $s>\frac34$ and $\epsilon$, $T$ satisfying $0<\epsilon T\le 1$. 
This implies estimate arguing as in \eqref{smoothing2b} that 
\begin{displaymath} 
\|P_Nv\|_{L^{2}_xL^{\infty}_T} \lesssim
\epsilon^{-s}\langle N \rangle^s\|P_Nv\|_{X^{0,\frac12,1}_{\epsilon,T}}.
\end{displaymath}
for any $N \ge 0$ and $s>\frac34$, which leads to \eqref{maximal1}  by squaring and summing over $N$.
\end{proof}

\subsection{Fractional Leibniz's rules}
First we state the classical fractional Leibniz rule estimate
derived by Kenig, Ponce and Vega (See Theorems A.8 and A.12 in \cite{KPV2}).
\begin{proposition} \label{leibrule}
Let $0<\alpha<1$, $p, \ p_1, \ p_2 \in (1,+\infty)$ with $\frac1{p_1}+\frac1{p_2}=\frac1p$
and $\alpha_1, \ \alpha_2 \in [0,\alpha]$ with $\alpha=\alpha_1+\alpha_2$. Then,
\begin{equation} \label{leibrule1}
\big\|D^{\alpha}_x(fg)-fD^{\alpha}_xg-gD^{\alpha}_xf \big\|_{L^p}
\lesssim \|D_x^{\alpha_1}g\|_{L^{p_1}}\|D^{\alpha_2}_xf\|_{L^{p_2}}.
\end{equation}
Moreover, for $\alpha_1=0$, the value $p_1=+\infty$ is allowed.
\end{proposition} 

The next estimate is a frequency localized version of estimate \eqref{leibrule1}, 
proved in \cite{Mo}, in the same spirit as Lemma 3.2 in \cite{Ta}. 
\begin{lemma} \label{lemma2}
Let $\alpha \ge 0$ and $1<q<\infty$. Then,
\begin{equation} \label{lemma2.1}
\big\|D_x^{\alpha}P_+\big(fP_-\partial_xg\big) \big\|_{L^q} \lesssim
\|D_x^{{\alpha}_1}f\|_{L^{q_1}}\|D_x^{{\alpha}_2}g\|_{L^{q_2}},
\end{equation}
with $1<q_i<\infty$, $\frac1{q_1}+\frac1{q_2}=\frac1q$ and $\alpha_1
\ge \alpha$, $\alpha_2 \ge 0$ and $\alpha_1+\alpha_2=1+\alpha$.
\end{lemma}

We also state an estimate to handle the multiplication by a
term on the form $e^{\pm\frac{i}{2}F}$, where $F$ is a real-valued function,
in fractional Sobolev spaces.
\begin{lemma} \label{lemma1}
Let $2 \le q <\infty$ and $1 \le s \le \frac32$. Consider $F$ and
$F_1$ two real-valued functions such that $v=\partial_xF$ and $v_1=\partial_xF_1$ belong
to $L^2(\mathbb R)$. Then, it holds that
\begin{equation} \label{lemma1.1}
\|J^s_x\big(e^{\pm iF} g\big)\|_{L^q} \lesssim
(1+\|v\|_{H^1}^2)\|J^s_xg\|_{L^q}.
\end{equation}
\end{lemma}

\begin{remark} The proof follows the lines of Lemma 2.7 in \cite{MP} (see also \cite{Mo, Mo1}). A version of Lemma \ref{lemma1} could also be stated for $s >\frac{3}{2}$.
\end{remark}

\section{The gauge transformation} \label{gt}
The gauge transform we will use is  the one introduced by Tao in \cite{Ta}. 
First we  define an antiderivative $F=F[v]$ of $ v$. 
We  determine $F$ on the time axis $x=0$ by solving the ODE
\begin{displaymath}
\left\{\begin{array}{l}
\partial_tF(0,t)=\big(b\mathcal{H}v_x+a\epsilon v_{xx}+\frac{c}2A^{-1}v^2-\frac32a\epsilon
(v\mathcal{H}v_x+\mathcal{H}(vv_x))\big)(0,t), \\
F(0,0)=0, \end{array} \right.
\end{displaymath}
Then we extend $ F $ on the whole plan by setting 
\begin{displaymath}
\partial_xF=Av, \quad \text{where} \quad A=\frac{2d}{3a}.
\end{displaymath}
Clearly, it holds 
\begin{displaymath}
\partial_x\big(\partial_tF-b\mathcal{H}\partial^2_xF-a\epsilon \partial_x^3F\big)=
\partial_x\big(\frac{c}2A^{-1}F_x^2-\frac32a\epsilon(F_x\mathcal{H}F_{xx}+\mathcal{H}(F_xF_{xx}))\big).
\end{displaymath}
and, according to  the choice of $ F $ on the time axis,  it satisfies the equation
\begin{equation} \label{F}
\partial_tF-b\mathcal{H}\partial^2_xF-a\epsilon \partial_x^3F=
\frac{c}2A^{-1}F_x^2-\frac32a\epsilon(F_x\mathcal{H}F_{xx}+\mathcal{H}(F_xF_{xx})).
\end{equation}

Now,  we
perform the following nonlinear transformation
\begin{equation} \label{w}
W=P_{+hi}(e^{iF}) \quad \text{and} \quad w=W_x=iAP_{+hi}(e^{iF}v).
\end{equation}
First, using the identity $\mathcal{H}P_+=-iP_+$, we compute
\begin{eqnarray*}
\lefteqn{\partial_tW+ib\partial^2_xW-a\epsilon \partial_x^3W} \\
&=iP_{+hi}\big(e^{iF}(\partial_tF+ib\partial^2_xF-a\epsilon\partial_x^3F
-bF_x^2-3ia\epsilon F_xF_{xx}-a\epsilon F_x^3)\big).
\end{eqnarray*}
Then using \eqref{F} and the identity
$\mathcal{H}+i=2iP_-$ it follows that 
\begin{eqnarray*}
\lefteqn{\partial_tW+ib\partial^2_xW-a\epsilon \partial_x^3W} \\
&&=P_{+hi}\big(e^{iF}(i(\frac{c}{2}A^{-1}-b)F_x^2-ia\epsilon
F_x^3-2bP_-F_{xx}+3a\epsilon F_xP_-F_{xx}+3a\epsilon
P_-(F_xF_{xx}))\big)
\\ &&
=P_{+hi}(e^{iF}\big(\alpha_1v^2+\alpha_2\epsilon
v^3)\big)+\alpha_3P_{+hi}(WP_-v_x)+\alpha_3P_{+hi}(P_{lo}(e^{iF})P_-v_x)
\\ && \quad
+\alpha_4\epsilon P_{+hi}(wP_-v_x)+\alpha_5 \epsilon
P_{+hi}(P_{lo}(e^{iF}v)P_-v_x)+\alpha_6 \epsilon P_{+hi}(WP_-(vv_x))
\\ && \quad +\alpha_6\epsilon P_{+hi}(P_{lo}(e^{iF})P_-(vv_x)),
\end{eqnarray*}
where $\alpha_j, \ j=1 \cdots 6$ are complex constants depending on $a$, $b$, $c$ and $d$.
\begin{remark} \label{condition}
We observe from the definition of the coefficients $a$, $b$, $c$ and
$d$ in \eqref{ab} and \eqref{cd} that
\begin{equation} \label{condition1}
\alpha_1=0 \quad \Leftrightarrow \quad \frac{3ac}{4d}=b \quad
\Leftrightarrow \quad \rho^2=3\rho_1^2.
\end{equation}
\end{remark}

In the following, we will fix $\alpha_1=\cdots \alpha_6=1$ for sake
of simplicity. Therefore, we deduce by differentiating the above
equation that $w$ is a solution to
\begin{equation} \label{eqw}
\begin{split}
\partial_tw+ib\partial^2_xw-a&\epsilon \partial_x^3w=\partial_xP_{+hi}(e^{iF}(v^2+\epsilon v^3)) \\
&+\partial_xP_{+hi}(WP_-v_x)+\partial_xP_{+hi}(P_{lo}(e^{iF})P_-v_x)\\
&+\epsilon\partial_xP_{+hi}(wP_-v_x)+\epsilon\partial_xP_{+hi}(P_{lo}(e^{iF}v)P_-v_x)
\\ &+\epsilon\partial_xP_{+hi}(WP_-(vv_x))
+\epsilon\partial_xP_{+hi}(P_{lo}(e^{iF})P_-(vv_x)) \\ 
& \quad \quad \ := N(e^{iF},v,W,w).
\end{split}
\end{equation}
On the other hand, we can recover $v$ as a function of $w$ by writing
\begin{equation} \label{v}
iAv=e^{-iF}\partial_x(e^{iF})=e^{-iF}w+e^{-iF}\partial_xP_{lo}(e^{iF})+e^{-iF}\partial_xP_{-hi}(e^{iF}),
\end{equation}
so that it follows from the frequency localization
\begin{equation} \label{P+v}
\begin{split}
iAP_{+HI}v&=P_{+HI}(e^{-iF}w)+P_{+HI}(P_{+hi}e^{-iF}\partial_xP_{lo}(e^{iF})) \\ & \quad+P_{+HI}(P_{+HI}e^{-iF}\partial_xP_{-hi}(e^{iF})).
\end{split}
\end{equation}

Then, we have the following \textit{a priori} estimates on $v$ in
terms of $w$.
\begin{proposition} \label{apriori v}
Let $s \ge 1$, $0<T\le1$, $0\le\theta\le 1$, $0<\epsilon \le 1$ and $v$ be a solution to
\eqref{hoBO} in the time interval $[0,T]$. Then, it holds that
\begin{equation} \label{apriori v.1}
\|v\|_{X^{s-2\theta,\theta}_{\epsilon,T}} \lesssim \|v\|_{L ^{\infty}_TH^s_x}+
\|v\|_{L^{\infty}_TH^s_x}^2+\epsilon \|J^s_xv\|_{L^4_{T,x}}^2 .
\end{equation}
Moreover, if $1 \le s \le \frac32$, it holds that
\begin{equation} \label{apriori v.2}
\|J^s_xv\|_{L^{\infty}_TL^2_x} \lesssim \|v_0\|_{H^1}
+\big(1+\|v\|_{L^{\infty}_TH^1_x}^2\big)\big(\|w\|_{X^{s,
\frac12,1}_{\epsilon,T}}+\|v\|_{L^{\infty}_TH^1_x}^2\big),
\end{equation}
\begin{equation} \label{apriori v.3}
\|J^s_xv\|_{L^4_{x,t}} \lesssim \|v_0\|_{H^1}
+\big(1+\|v\|_{L^{\infty}_TH^1_x}^2\big)\big(\epsilon^{-\frac{1}{12}}\|w\|_{X^{s,
\frac13+}_{\epsilon,T}}+\|v\|_{L^{\infty}_TH^1_x}^2\big),
\end{equation}
\begin{equation} \label{apriori v.13}
\|v \|_{L^{2}_x L^\infty_T} \lesssim \varepsilon^{-1} \Bigr( \|v_0\|_{H^1}+ \|v\|_{L^{\infty}_TH^1_x}(
 \|w\|_{X^{1,1/2,1}_{\epsilon,T}} +   \|v\|_{L^{\infty}_TH^1_x} + \|v\|_{L^2_x L^\infty_T})
\Bigl),
\end{equation}
and 
\begin{equation} \label{apriori v.14b} 
\begin{split}
\|J^s_x\partial_x v\|_{\widetilde{L^{\infty}_x L^2_T}} & \lesssim   \Big(\epsilon^{-\frac32+}\|v_0\|_{H^s}+
\epsilon^{-\frac32+}\|w\|_{X^{s,\frac12,1}_{\epsilon,T}}\\ & \quad \ +\|v\|_{L^{\infty}_TH^1_x}\big(\|w\|_{X^{1,\frac12,1}_{\epsilon,T}}+\|J_x^s \partial_xv \|_{\widetilde{L^\infty _xL^{2}_T}}+\|v\|_{L^{\infty}_TH^1_x}\big)\Big).
\end{split}
\end{equation}
\end{proposition}

\begin{remark}  It is worth noticing that estimates \eqref{apriori v.2} and \eqref{apriori v.3} could be rewritten in a convenient form for $s>\frac32$.
\end{remark}

\begin{proof}  \eqref{apriori v.2} and \eqref{apriori v.3} follow from \eqref{P+v} for the high frequencies and from \eqref{hoBO} for the low frequencies (see for instance \cite{MP}).  To prove  \eqref{apriori v.1} we proceed  as in  \cite{MP}, noticing that according to  \eqref{hoBO}, 
$$
\| \partial_t (V_\varepsilon (-t) u(t) ) \|_{L^2_T H^{s-2}_x} \lesssim \| J^s_x u \|_{L^4_{Tx}}^2 \; .
$$
To prove estimate \eqref{apriori v.13}, we also split $v$ between its high and low Fourier modes 
\begin{equation} \label{apriori v.14}
\|v\|_{\widetilde{L^2_xL^{\infty}_T}} \lesssim
\|P_{LO}v\|_{\widetilde{L^2_xL^{\infty}_T}} +\|P_{HI}v\|_{\widetilde{L^2_xL^{\infty}_T}} .
\end{equation} 
The low frequency  term on the right-hand side of \eqref{apriori v.14} can be treated by using \eqref{hoBO} and the maximal function estimate \eqref{maximal5} to get 
\begin{equation} \label{apriori v.15}
\|P_{LO}v\|_{\widetilde{L^2_xL^{\infty}_T}} \lesssim
\epsilon^{-(\frac34+)}\left(\|v_0\|_{H^1}+\|v\|_{L^{\infty}_TH^1_x}^2\right).
\end{equation} 
To treat the high frequency term we use that $ v $ is real-valued to first notice that 
$$
\|P_{HI} v \|_{L^2_x L^\infty_T} \lesssim 2 \| P_{+HI} v \|_{L^2_x L^\infty_T} 
$$
so that we are reduced to estimate each terms on the right-hand side of \eqref{P+v}. Now the problem is that $ P_{+HI} $ is not continuous in $ L^2_x L^\infty_T $. We will overcome this difficulty  by noticing  that  $ P_{+HI}=\sum_{k\ge 4} P_{+ 2^k} $ and that the family of operators $  P_{+ 2k} $ is bounded in $ L^2_x L^\infty_T $. To treat  the first term of the  right-hand side of \eqref{P+v} we first notice that for $ k\ge 4 $, 
\begin{equation}
P_{+2^k} (e^{-iF} w)=P_{+2^k}\Bigl(  \sum_{j\ge k-3} P_{2^j} w P_{\le 2^{j-1}} (e^{-iF})\Bigr) + P_{+2^k}\Bigl(  \sum_{j\ge k-3} P_{\le 2^{j}}
w P_{ 2^{j}} (e^{-iF})\Bigr)  \label{tyty}
\end{equation}
so that 
\begin{eqnarray*}
\|P_{+HI} (e^{-iF} w)  \|_{L^2_x L^\infty_T} & \lesssim &  \sum_{k\ge 4} \|P_{+2^k } (e^{-iF} w)  \|_{L^2_x L^\infty_T} \\
& \lesssim & \sum_{k\ge 4} \Bigl[ \sum_{j\ge k-3} \|P_{2^j  } w \|_{L^2_x L^\infty_T} \\ && + 
\sum_{j\ge k-3} \|P_{ 2^{j}} (e^{-iF})\|_{L^\infty_{T,x}} \|P_{\le 2^j} w\|_{L^2_x L^\infty_T}  \Bigr].
\end{eqnarray*}
But on one hand, for $ s\ge 1 $ we deduce from \eqref{maximal1} and Bernstein inequalities  that for $ \alpha\in ]0,1/4[ $, 
\begin{displaymath}
\begin{split}
\sum_{k\ge 4} \sum_{j\ge k-3} \|P_{+2^j  } w \|_{L^2_x L^\infty_T}& \lesssim \sum_{k\ge 4} \sum_{j\ge k-3} 2^{-\alpha j} 
\|D_x^\alpha P_{+2^j  } w \|_{L^2_x L^\infty_T} \\ &
\lesssim \sup_{k\ge 3} \| D_x^\alpha P_{2^k}  w \|_{L^2_x L^\infty_T}
\lesssim \varepsilon ^{-1} \|w\|_{X^{1,1/2,1}_T}
\end{split}
\end{displaymath}
and on the other hand, 
\begin{eqnarray*}
\sum_{k\ge 4} \sum_{j\ge k-3} \|P_{ 2^{j}} (e^{-iF})\|_{L^\infty_{Tx}} \|P_{\le 2^j} w\|_{L^2_x L^\infty_T} & \lesssim  & 
\sum_{k\ge 4} \sum_{j\ge k-3} 2^{-j} \| v \|_{L^\infty_{Tx}} \|P_{\le 2^j} w\|_{L^2_x L^\infty_T}\\
& \lesssim &  \| v\|_{L^\infty_{T}H^1_x} \sup_{k\ge 1} \| P_{2^k}  w \|_{L^2_x L^\infty_T} \\ &\lesssim & \varepsilon ^{-1} \|w\|_{X^{1,1/2,1}_T} 
\| v\|_{L^\infty_{T}H^1_x},
\end{eqnarray*}
which completes the estimate of the term $P_{+HI} (e^{-iF} w)$. To treat the second term and third terms  of the  right-hand side of \eqref{P+v} we proceed as above, using the frequency localization due to the projections,  to obtain 
\begin{eqnarray} \label{apriori v.17} 
\big\| P_{+HI} \big(e^{-iF}P_{-hi}(e^{iF}v)\big)\big\|_{L^2_xL^{\infty}_T}  &\lesssim & 
\sum_{N \ge 2^4} \sum_{K \ge N/4} \|P_{K}e^{iF}\|_{L^{\infty}_{x,T}}\|P_{\le K}P_{-hi}(e^{iF}v)\|_{L^2_xL^{\infty}_T}\Big) \nonumber\\
& \lesssim  & 
\sum_{N \ge 2^4} \Bigr( \sum_{K \ge N/4} K^{-1} \|v\|_{L^{\infty}_{x,T}} \sum_{2\le Q\le K} \|P_{+Q}v\|_{L^2_xL^{\infty}_T}\Bigr)\nonumber\\
 & \lesssim & \|v\|_{L^{\infty}_TH^1_x}\|v\|_{L^2_xL^{\infty}_T}.\label{apriori v.1717} 
\end{eqnarray}
which completes the proof of \eqref{apriori v.13}.  

Finally, to prove \eqref{apriori v.14b} we proceed similarly. First we use \eqref{hoBO} and Sobolev inequality 
 to get 
 $$
 \|P_{LO}v\|_{\widetilde{L^\infty_x L^{2}_T}} \lesssim
\left(\|v_0\|_{H^1}+\|v\|_{L^{\infty}_TH^1_x}^2\right).
 $$
Second, from \eqref{tyty} we deduce that  for any  integer $ k\ge 4 $,  
 \begin{displaymath}
 \begin{split}
\|J^s_x& \partial_x P_{+2^k} (e^{-iF} w)  \|_{L^\infty_x L^2_T} \\ & \lesssim  2^{k(s+1)} \Bigl(  \sum_{j\ge k-3} \|P_{2^j } w  \|_{L^\infty_x L^2_T}
+  \sum_{j\ge k-3} \|P_{\le 2^{j} } w  \|_{L^\infty_{T,x}} \|P_{2^j} e^{-iF} \|_{L^\infty_{x}L^2_T }\Bigr) \\
& \lesssim  \sum_{j\ge k-3}  2^{(k-j)(s+1)} \|J^s_x \partial_x P_{2^j  } w \|_{L^\infty_x L^2_T} + 
 \| w\|_{L^2_T H^1_x}\sum_{j \ge k-3}2^{-j} \|J^s_x \partial_x  P_{2^{j}} v\|_{L^\infty_{x}L^2_T } . 
\end{split}
\end{displaymath}
Therefore, noticing that the first term on the right-hand side of the above inequality is a discrete convolutions between $  \{ 2^{q(s+1)}\}_{q\le 1} \in l^1(\Z) $ and  
 $\{  \|J^s_x \partial_x P_{2^j  } w \|_{L^\infty_x L^2_T}\}_{j\in \Z_+}$, we deduce from Young's inequality that 
 \begin{eqnarray*}
  \|J^s_x \partial_x P_{+HI} (e^{-iF} w)  \|_{\widetilde{L^\infty_x L^2_T}}  & \lesssim  & \|J^s_x \partial_x w  \|_{\widetilde{L^\infty_x L^2_T}}
 +  \| w\|_{L^2_T H^1_x}  \sup_{j \ge 3}\|J_x^s \partial_x v\|_{L^\infty_xL^{2}_T}  \\
  & \lesssim & \epsilon^{-\frac32+}\|w\|_{X^{s,1/2,1}_{\epsilon,T}}  + \|w\|_{X^{1,1/2,1}_{\epsilon,T}} \|J_x^s \partial_x v\|_{\widetilde{L^\infty_xL^{2}_T}} ,
 \end{eqnarray*}
 where we made use of \eqref{smoothing4} in the last step. 
 Third, we proceed similarly to estimate   the second and third  term of the  right-hand side of \eqref{P+v} by 
 $$
 \big\| J^s_x \partial_x P_{+HI} \big(e^{-iF}P_{lo}(e^{iF}v)\big)\big\|_{\widetilde{L^\infty_xL^{2}_T}}
+\big\| J^s_x \partial_x P_{+HI} \big(e^{-iF}P_{-hi}(e^{iF}v)\big)\big\|_{\widetilde{L^\infty_xL^{2}_T}}  $$
$$\lesssim  
 \|v\|_{L^\infty_T  H^1_x}  \sup_{j\ge 3 } \|J^s_x \partial_x  P_{2^{j}} v\|_{L^\infty_{x}L^2_T }.
 $$
 This completes the proof of \eqref{apriori v.14b} and of the proposition.
\end{proof}

\section{Bilinear estimates}
In this section, we fix $\epsilon=1$. The aim of this section
is to derive an estimate on $\|w\|_{X^{s,\frac12,1}_T}$.
\begin{proposition} \label{apriori w}
Let $0<T\le 1$, $1 \le s \le \frac32$, $v$ be a solution to \eqref{hoBO} on
the time interval $[0,T]$ and $w$ defined in \eqref{w}. Then it holds that
\begin{equation} \label{apriori w.1}
\begin{split}
&\|w\|_{X^{s,\frac12,1}_T} \lesssim \big(1+\|v_0\|_{H^1}^2
\big)\|v_0\|_{H^s} \\ & +p\big(\|w\|_{X^{s,\frac12,1}}, \|v\|_{L^{\infty}_TH^1_x}, \|v\|_{L^4_TW^{1,4}_x},\|v\|_{L^2_xL^{\infty}_T}, 
\|J^s_x\partial_x v \|_{\widetilde{L^\infty_xL^{2}_T}},
\sup_{0\le \theta \le 1}\|v\|_{X^{1-2\theta,\theta}_T} \big),
\end{split}
\end{equation}
where $p$ is a polynomial function  at least quadratic in its arguments. 
\end{proposition}

The main tools to  prove Proposition \ref{apriori w} are  the
following crucial bilinear estimates.
\begin{proposition} \label{bilin}
For any  $s \ge 1$, we have that
\begin{equation} \label{bilin1}
\|\partial_xP_{+hi}\big(wP_-\partial_xv\big)\|_{X^{s,-\frac12,1}} \lesssim
\|w\|_{X^{s,\frac12,1}}\sup_{0
\le \theta \le 1}\|v\|_{X^{1-2\theta,\theta}}.
\end{equation}
\end{proposition}

\begin{proof}  We only prove estimate \eqref{bilin1} in the case $s=1$, since the case $s>1$ follows by similar arguments due to the frequency localization on the functions $w$ and $v$.  By duality, estimate \eqref{bilin1} is equivalent to
\begin{equation} \label{bilin2}
\big|I\big| \lesssim \Big(\sum_N\sup_L\|h_{N,L}\|_{L^2_{\xi,\tau}}^2\Big)^{\frac12}\|w\|_{X^{1,\frac12,1}}\sup_{0
\le \theta \le 1}\|v\|_{X^{1-2\theta,\theta}},
\end{equation}
where
\begin{equation} \label{bilin3}
I=\sum_{N,L}\langle N \rangle \langle L \rangle^{-\frac12}
\int_{\mathcal{D}}\xi h_{N,L}(\xi,\tau)
\phi_N(\xi)\psi_L(\xi,\tau)
\widehat{w}(\xi_1,\tau_1)\xi_2 \widehat{v}(\xi_2,\tau_2)d\nu,
\end{equation}
\begin{equation} \label{bilin4}
d\nu=d\xi d\xi_1 d\tau d\tau_1, \quad \xi_2=\xi-\xi_1, \quad \tau_2=\tau-\tau_1,
\end{equation}
\begin{equation} \label{bilin5}
\sigma=\tau-|\xi|\xi+\xi^3, \quad \sigma_i=\tau_i-\xi_i|\xi_i|+\xi_i^3, \ i=1,2,
\end{equation}
and
\begin{equation} \label{bilin6}
\mathcal{D}=\big\{(\xi,\xi_1,\tau,\tau_1) \in \mathbb R^4 \ | \ \xi \ge 1, \ \xi_1 \ge 1 \ \text{and} \ \xi_2 \le 0 \big\}.
\end{equation}
Observe that we always have in $\mathcal{D}$ that
\begin{equation} \label{bilin7}
\xi_1 \ge \xi \ge 1 \quad \text{and} \quad \xi_1 \ge |\xi_2|.
\end{equation}

Then, we obtain by performing dyadic decompositions in $\xi_1$, $\xi_2$, $\sigma_1$ and $\sigma_2$ that
\begin{equation} \label{bilin8}
\begin{split}
I=\sum_{N,N_,N_2}\sum_{L,L_1,L_2}\langle N \rangle &\langle L \rangle^{-\frac12}
\int_{\mathcal{D}} \xi h_{N,L}(\xi,\tau)
\phi_N(\xi)\psi_L(\xi,\tau)
\\ &
\times (P_{N_1}Q_{L_1}w)^{\wedge}(\xi_1,\tau_1)\xi_2 (P_{N_2}Q_{L_2}v)^{\wedge}(\xi_2,\tau_2)d\nu.
\end{split}
\end{equation}
Due to the second identity in \eqref{bilin4} and \eqref{bilin7}, we can always assume that one of the following cases holds:
\begin{enumerate}
\item{}  high-low interaction: $N_1 \sim N$ and $N_2 \le N_1$
\item{} high-high interaction: $N_1 \sim N_2$ and $N \le N_1$.
\end{enumerate}
Moreover, the resonance identity
\begin{equation} \label{bilin9}
\sigma-\sigma_1-\sigma_2=3\xi\xi_2(\xi_1-\frac23)
\end{equation}
holds in $\mathcal{D}$, so that for fixed $N$, $N_1$ and $N_2$, we can always assume that
\begin{equation} \label{bilin10}
L_{max} \sim \max\{L_{med},NN_1N_2\},
\end{equation}
where $L_{max}$, $L_{med}$ and $L_{min}$ denote respectively the maximum, median and minimum of $L$, $L_1$ and $L_2$.

To estimate $I$, we will divide the sum in \eqref{bilin8} depending on the high-low or high-high interactions regime and on wether $L_{\max}=L,\ L_1$ or $L_2$.  \\

\noindent \textit{Case high-low interaction and $L_{\max}=L_2$.} In this case, we can estimate $I$ as
\begin{displaymath}
\begin{split}
|I| \lesssim \sum_{N_1, N_2 \le N_1}\sum_{L_2, L \le L_2}\frac{\langle N_1 \rangle \langle L \rangle^{-\delta}}{\langle L_2 \rangle}
&\int_{\mathcal{D}} \frac{|\xi\xi_2|\langle \xi_2\rangle}{\langle \xi_1 \rangle} \frac{|h_{N_1,L}(\xi,\tau)|}{\langle \sigma \rangle^{\frac12-\delta}}
\phi_{N_1}(\xi)\psi_L(\xi,\tau) \\ &
\times |(P_{N_1}J_x^1w)^{\wedge}(\xi_1,\tau_1)| \frac{\langle \sigma_2 \rangle}{\langle \xi_2 \rangle} |(P_{N_2}Q_{L_2}v)^{\wedge}(\xi_2,\tau_2)|d\nu.
\end{split}
\end{displaymath}
Since $L_2=L_{max}$, we deduce from \eqref{bilin9} that $L_2 \ge N_1^2N_2$. Therefore, $L_2 \sim 2^kN_1^2N_2$ for $k \in \mathbb Z_+$, so that, we obtain by using Plancherel's identity and H\"older's inequality
\begin{displaymath}
\begin{split}
|I| &\lesssim \sum_{N_1, N_2\le N_1}\frac{N_2}{N_1}
\sup_{L}\big\| \Big(\frac{| h_{N_1,L}|}{\langle\sigma\rangle^{\frac12-\delta}}\Big)^{\vee}\big\|_{L^4_{x,t}}
\big\|\big(\big|(P_{N_1}J_x^1w)^{\wedge} \big|\big)^{\vee}\big\|_{L^4_{x,t}} \\ & \quad \quad \quad \quad \times\sum_{k \in \mathbb Z_+}2^{-k}\|\frac{\langle \sigma \rangle}{\langle \xi \rangle}(P_{N_2}Q_{2^kN_1^2N_2}v)^{\wedge}\|_{L^2_{\xi,\tau}}. \\
& \lesssim \sum_{N_1}\sup_{L}\|h_{N_1,L}\|_{L^2_{\xi,\tau}}
\big\|\big(\big|(P_{N_1}J_x^1w)^{\wedge} \big|\big)^{\vee}\big\|_{L^4_{x,t}}\|v\|_{X^{-1,1}} \\ &
\lesssim \Big(\sum_{N_1}\sup_{L}\|h_{N_1,L}\|_{L^2_{\xi,\tau}}^2 \Big)^{\frac12}
\Big(\sum_{N_1}\big\|\big(\big|(P_{N_1}J_x^1w)^{\wedge} \big|\big)^{\vee}\big\|_{L^4_{x,t}}^2 \Big)^{\frac12}\|v\|_{X^{-1,1}},
\end{split}
\end{displaymath}
which, combined to estimates \eqref{strichartz7} and \eqref{propBourgain1} leads to estimate \eqref{bilin2} in this case. Note that we have used here any $0<\delta<\frac16-$, since $X^{0,\frac13+} \hookrightarrow L^4_{x,t}$ due to estimate \eqref{strichartz7}. \\

\noindent \textit{Case high-high interaction and $L_{\max}=L_2$.} This case works exactly as in the precedent case, summing in $N \le N_1$ instead of $N_2 \le N$.  \\

\noindent \textit{Case high-low interaction and $L_{\max}=L_1$.}  We obtain from the frequency localization properties that
\begin{displaymath}
\begin{split}
|I| \lesssim \sum_{N_1, N_2 \le N_1}\sum_{L_1, L \le L_1}&\frac{\langle N_1 \rangle \langle L \rangle^{-\delta}\langle L_1 \rangle^{\frac12}}{\langle L_1 \rangle^{\frac12}}
\int_{\mathcal{D}} |\xi\xi_2^{\frac12-\delta_2}| \frac{|h_{N_1,L}(\xi,\tau)|}{\langle \sigma \rangle^{\frac12-\delta}}
\phi_{N_1}(\xi)\psi_L(\xi,\tau) \\ &
\times |(P_{N_1}Q_{L_1}w)^{\wedge}(\xi_1,\tau_1)|  |(P_{N_2}D_x^{\frac12+\delta_2}v)^{\wedge}(\xi_2,\tau_2)|d\nu,
\end{split}
\end{displaymath}
in this case.  Therefore, it follows from \eqref{bilin9}, estimate \eqref{strichartz7}, Plancherel's identity and H\"older's inequality that
\begin{equation} \label{bilin11}
\begin{split}
|I| &\lesssim \sum_{N_1, N_2\le N_1}N_2^{-\delta_2}
\sup_{L}\big\| \Big(\frac{| h_{N_1,L}|}{\langle\sigma\rangle^{\frac12-\delta}}\Big)^{\vee}\big\|_{L^6_{x,t}} \\ & \quad \times
\sum_{L_1}\langle N_1\rangle \langle L_1 \rangle^{\frac12}\|P_{N_1}Q_{L_1}w\|_{L^2_{x,t}}\|\big(|(P_{N_2}D^{\frac12+\delta_2}_xv)^{\wedge}|\big)^{\vee}\|_{L^3_{x,t}}\\  &
\lesssim \Big(\sum_{N_1}\sup_{L}\big\| \Big(\frac{| h_{N_1,L}|}{\langle\sigma\rangle^{\frac12-\delta}}\Big)^{\vee}\big\|_{L^6_{x,t}}^2\Big)^{\frac12}
\|w\|_{X^{1,\frac12,1}}\Big(\sum_{N_2}\|\big(|(P_{N_2}D^{\frac12+\delta_2}_xv)^{\wedge}|\big)^{\vee}\|_{L^3_{x,t}}^2\Big)^{\frac12},
\end{split}
\end{equation}
where $\delta$ and $\delta_2$ are two small positive numbers. Now, we have from \eqref{strichartz7} that the injection
$X^{0,\frac29+} \hookrightarrow L^3_{x,t}$ is continuous. Then, we can choose $\delta_2$ postive small enough such that
\begin{equation} \label{bilin12}
\Big(\sum_{N_2}\|\big(|(P_{N_2}D^{\frac12+\delta_2}_xv)^{\wedge}|\big)^{\vee}\|_{L^3_{x,t}}^2\Big)^{\frac12} \lesssim \|v\|_{X^{\frac12+\delta_2,\frac14-\frac{\delta_2}2}} \lesssim \sup_{0 \le \theta \le 1}
\|v\|_{X^{1-2\theta,\theta}}.
\end{equation}
Estimate \eqref{strichartz7} also implies $X^{0,\frac49+} \hookrightarrow L^6_{x,t}$, so that
\begin{equation} \label{bilin13}
\Big(\sum_{N_1}\sup_{L}\big\| \Big(\frac{| h_{N_1,L}|}{\langle\sigma\rangle^{\frac12-\delta}}\Big)^{\vee}\big\|_{L^6_{x,t}}^2\Big)^{\frac12} \lesssim  \Big(\sum_{N_1}\sup_{L}\| h_{N_1,L}\|_{L^2_{\xi,\tau}}^2\Big)^{\frac12},
\end{equation}
for $0<\delta$ and small enough. We deduce estimate \eqref{bilin2} in this case gathering \eqref{bilin11}--\eqref{bilin13}. \\

\noindent \textit{Case high-high interaction and $L_{\max}=L_1$.}  We proceed exactly as in the precedent case. $I$ can be estimate as
\begin{displaymath}
\begin{split}
|I| &\lesssim \sum_{N_1, N \le N_1}\sum_{L_1, L \le L_1}\frac{\langle N \rangle \langle L \rangle^{-\delta}\langle L_1 \rangle^{\frac12}}{\langle L_1 \rangle^{\frac12}}
\int_{\mathcal{D}} |\xi\xi_2^{\frac12}| \frac{|h_{N,L}(\xi,\tau)|}{\langle \sigma \rangle^{\frac12-\delta}}
\phi_{N}(\xi)\psi_L(\xi,\tau) \\ &
\quad \quad \quad \quad \quad \quad \quad \quad \times |(P_{N_1}Q_{L_1}w)^{\wedge}(\xi_1,\tau_1)|  |(P_{N_1}D_x^{\frac12}v)^{\wedge}(\xi_2,\tau_2)|d\nu \\ &
\lesssim \sum_{N_1, N\le N_1}\big(\frac{N}{N_1}\big)^{\frac12}
\sup_{L}\big\| \Big(\frac{| h_{N,L}|}{\langle\sigma\rangle^{\frac12-\delta}}\Big)^{\vee}\big\|_{L^6_{x,t}} \\ & \quad \quad \quad \quad \quad \quad \quad \quad \times
\sum_{L_1}\langle N_1\rangle \langle L_1 \rangle^{\frac12}\|P_{N_1}Q_{L_1}w\|_{L^2_{x,t}}\|\big(|(P_{N_1}D^{\frac12}_xv)^{\wedge}|\big)^{\vee}\|_{L^3_{x,t}} \\  &
\lesssim \Big(\sum_{N}\sup_{L}\big\| \Big(\frac{| h_{N,L}|}{\langle\sigma\rangle^{\frac12-\delta}}\Big)^{\vee}\big\|_{L^6_{x,t}}^2\Big)^{\frac12}
\|w\|_{X^{1,\frac12}}\Big(\sum_{N_1}\|\big(|(P_{N_1}D^{\frac12+\delta_2}_xv)^{\wedge}|\big)^{\vee}\|_{L^3_{x,t}}^2\Big)^{\frac12},
\end{split}
\end{displaymath}
which yields estimate \eqref{bilin2} in this case, recalling \eqref{bilin12} and \eqref{bilin13}.  \\

\noindent \textit{Case high-low interaction and $L_{\max}=L$.} From \eqref{bilin10}, we can always assume that $L \sim (NN_1N_2)^{\frac12}$ in this case, since when $L \sim L_{med}$ we are in one of the precedent cases. Therefore, for fixed $N$, $N_1$ and $N_2$, we only have a finite number of terms in the sum in $L$ appearing on \eqref{bilin8}, so that the following estimate holds in this case,
\begin{displaymath}
\begin{split}
|I| &\lesssim \sum_{N_1, N_2 \le N_1}\frac{\langle N_1 \rangle}{(N_1^2N_2)^{\frac12}}
\int_{\mathcal{D}} |\xi\xi_2^{\frac12-\delta_2}| |h_{N_1,L}(\xi,\tau)|\phi_{N_1}(\xi)\psi_L(\xi,\tau) \\ &
 \quad \quad \quad \quad \quad \quad \quad \quad \times |(P_{N_1}w)^{\wedge}(\xi_1,\tau_1)|  |(P_{N_2}D_x^{\frac12+\delta_2}v)^{\wedge}(\xi_2,\tau_2)|d\nu \\ &
 \lesssim \sum_{N_1, N_2\le N_1}N_2^{-\delta_2}
\sup_{L}\|  h_{N_1,L}\big\|_{L^2_{\xi,\tau}} \|\big(|(P_{N_1}J_x^1w)^{\wedge}|\big)^{\vee}\|_{L^6_{x,t}}\|\big(|(P_{N_2}D^{\frac12+\delta_2}_xv)^{\wedge}|\big)^{\vee}\|_{L^3_{x,t}} \\  &
\lesssim \Big(\sum_{N_1}\sup_{L}\|  h_{N_1,L}\|_{L^2_{\xi,\tau}}^2\Big)^{\frac12}
\|w\|_{X^{1,\frac12,1}}\|v\|_{X^{\frac12+\delta_2,\frac14-\frac{\delta_2}2}}.
\end{split}
\end{displaymath}
This proves estimate \eqref{bilin2} in this case. \\

\noindent \textit{Case high-high interaction and $L_{\max}=L$.} This case can be treated combining the ideas used for the \textit{high-high interaction and $L_{\max}=L_1$ case} and the \textit{high-low interaction and $L_{\max}=L$} case.
\end{proof}

\begin{proposition} \label{bilina} 
Let $0 < T \le 1$, $s \ge 1$, $v\in L^{\infty}(\mathbb R;L^2(\mathbb R)) \cap L^4(\mathbb R;W^{1,4}(\mathbb R))$ and $w \in X^{1,\frac12,1}$ supported in the time interval $[0,2T]$. Then, it holds that
\begin{equation} \label{bilina1} 
\|\partial_xP_{+hi}\big(WP_-v_x\big)\|_{X^{s,-\frac12,1}} \lesssim 
\|w\|_{X^{s,\frac12,1}}\|\partial_xv\|_{L^4_{x,t}},
\end{equation}
\begin{equation} \label{bilina2} 
\|\partial_xP_{+hi}\big(WP_-\partial_x(v^2)\big)\|_{X^{s,-\frac12,1}} \lesssim 
\|w\|_{X^{s,\frac12,1}}\|\partial_xv\|_{L^4_{x,t}}\|v\|_{L^{\infty}_tH^1_x},
\end{equation} 
and
\begin{equation} \label{bilina3}
\begin{split}
\|\partial_xP_{+hi}\big(P_{lo}(e^{iF})P_-v_x\big)&\|_{X^{s,-\frac12,1}} 
+\|\partial_xP_{+hi}\big(P_{lo}(e^{iF}v)P_-v_x\big)\|_{X^{s,-\frac12,1}} \\&
+\|\partial_xP_{+hi}\big(P_{lo}(e^{iF})P_-\partial_x(v^2)\big)\|_{X^{s,-\frac12,1}}\lesssim 
\|v\|_{L^4_{x,t}}^2.
\end{split}
\end{equation}
\end{proposition}

\begin{proof} 
We begin with the proof of estimate \eqref{bilina1}. We deduce from the Cauchy-Schwarz inequality and estimate \eqref{propBourgain1b} that
\begin{displaymath} 
\begin{split}
\|\partial_xP_{+hi}\big(WP_-v_x\big)\|_{X^{s,-\frac12,1}} & \lesssim 
\|\partial_xP_{+hi}\big(WP_-v_x\big)\|_{L^{1+}_tH^s_x} \\ 
& \lesssim \|\partial_xP_{+hi}\big(WP_-v_x\big)\|_{L^{1+}_tL^2_x}
+\|D_x^s\partial_xP_{+hi}\big(WP_-v_x\big)\|_{L^{1+}_tL^2_x}.
\end{split}
\end{displaymath}
Thus, it follows applying estimate \eqref{lemma2.1} and H\"older's inequality in time that 
\begin{displaymath} 
\|\partial_xP_{+hi}\big(WP_-v_x\big)\|_{X^{s,-\frac12,1}}  \lesssim 
T^{\frac12-}\big(\|w\|_{L^4_{x,t}}+\|D_x^sw\|_{L^4_{x,t}} \big)\|v_x\|_{L^4_{x,t}},
\end{displaymath}
which proves estimate \eqref{bilina1} since $X^{0,\frac12,1} \hookrightarrow X^{0,\frac13+} \hookrightarrow L^4_{x,t}$ by combining \eqref{propBourgain1} and \eqref{strichartz7}. 
Similar arguments combined to the Sobolev embedding $H^1(\mathbb R) \hookrightarrow L^{\infty}(\mathbb R)$ imply estimate \eqref{bilina2}. 

Finally, we turn to the proof of estimate \eqref{bilina3}. We will only bound the first term on the left-hand side, since the other ones can be treated exactly by the same way.   The Cauchy-Schwarz inequality and estimate \eqref{propBourgain1b} imply that
\begin{displaymath} 
\begin{split}
\|\partial_x&P_{+hi}\big(P_{lo}(e^{iF})P_-v_x\big)\|_{X^{s,-\frac12,1}} \\ &\lesssim 
 \|\partial_xP_{+hi}\big(P_{lo}(e^{iF})P_-v_x\big)\|_{L^{1+}_tL^2_x}
+\|D_x^s\partial_xP_{+hi}\big(P_{lo}(e^{iF})P_-v_x\big)\|_{L^{1+}_tL^2_x}.
\end{split}
\end{displaymath}
On the other hand, we get from the frequency localization that 
\begin{displaymath} 
\partial_xP_{+hi}\big(P_{lo}(e^{iF})P_-v_x\big)=\partial_xP_{+LO}\big(P_{lo}(e^{iF})P_{-LO}v_x\big).
\end{displaymath}
Therefore, we deduce from Bernstein's inequalities and estimate \eqref{lemma2.1} that 
\begin{displaymath} 
\begin{split}
\|\partial_xP_{+hi}\big(P_{lo}(e^{iF})P_-v_x\big)\|_{X^{s,-\frac12,1}} \lesssim 
T^{\frac12-}\|\partial_xe^{iF}\|_{L^4_{x,t}}\|v\|_{L^4_{x,t}} \lesssim \|v\|_{L^4_{x,t}}^2,
\end{split}
\end{displaymath}
recalling that $\partial_xF=Av$. This concludes the proof of Proposition \ref{bilina}.
\end{proof}

\begin{proposition} \label{bilinb}
Let $1 \le s \le \frac32$, $0 \le T \le 1$, $v$ a solution to \eqref{hoBO} which belongs to $L^{\infty}(\mathbb R;L^2(\mathbb R)) \cap L^2(\mathbb R; L^{\infty}[0,T])$, such that $J_x^s\partial_xv \in \widetilde{L^{\infty}_xL^{2}_T}$ and  supported in the time interval $[0,2T]$. Then, it holds that
\begin{equation} \label{bilinb1} 
\|\partial_xP_{+hi}\big(e^{iF}v^2\big)\|_{X^{s,-\frac12,1}} \lesssim 
p_1( \|v\|_{L^{\infty}_TH^1_x}) \|v\|_{L^{\infty}_TH^1_x} \|v\|_{L^{\infty}_TH^s_x} +
 \|v\|_{L^2_x L^\infty_T}\|J^s_x \partial_xv \|_{ \widetilde{L^\infty_x L^2_T}}  
\end{equation}
and
\begin{equation} \label{bilinb2} 
\|\partial_xP_{+hi}\big(e^{iF}v^3\big)\|_{X^{s,-\frac12,1}} \lesssim p_2 ( \|v\|_{L^{\infty}_TH^1_x}) \|v\|_{L^{\infty}_TH^1_x} \Bigl(\|v\|_{L^{\infty}_TH^s_x} +
 \|v\|_{L^2_x L^\infty_T}\|J^s_x\partial_xv \|_{ \widetilde{L^\infty_x L^2_T}} \Bigr)  ,
\end{equation} 
where $p_1$ and $p_2$ are polynomial  functions.
\end{proposition} 
\begin{proof}
First we notice that according to \eqref{lemma1.1}
\begin{eqnarray*}
\|\partial_xP_{+hi}\big(e^{iF}v^2\big)\|_{X^{s,-\frac12,1}} & \lesssim & \|\partial_xP_{+hi}\big(e^{iF}v^2\big)\|_{L^{1+}_T H^s}\\
& \lesssim & \| e^{iF} v^3\|_{L^{1+}_T  H^s_x}+\| e^{iF}\partial_x ( v^2) \big)\|_{L^{1+}_T H^s_x}\\
& \lesssim &(1+\|v\|_{L^\infty_T H^1_x}^2)( \|v\|_{L^\infty_T H^1_x}^2 \|v\|_{L^\infty_T H^s_x}+  \|\partial_x(v^2)\|_{L^2_T H^s_x}).
\end{eqnarray*}
It thus remains to control $ \|P_{HI} \partial_x(v^2)\|_{L^2_T H^s_x} $ since  the low frequency part of the contribution of $  \partial_x(v^2)$  is easily estimated by 
 $ \|v\|_{L^\infty_T H^1_x}^2 $. To this aim we use the same decomposition as in \eqref{tyty} and write for $ k\ge 4 $, 
$$
P_{2^k}( v^2)= P_{2^k}\Bigl( \sum_{j\ge k-3} P_{2^j} v P_{\le 2^{j-1}} v + \sum_{j\ge k-3 } P_{2^j} v P_{\le2^{j}} v\Bigr) .
$$
Hence,
\begin{eqnarray*}
  \|P_{2^k} \partial_x (v^2)\|_{L^2_T H^s_x} & \lesssim  & 2^{k(s+1)}
   \sum_{j\ge k-2} 2^{-j(s+1)}\|P_{2^j} J^s_x  v_{x} \|_{L^\infty_x L^2_T} \|v\|_{L^2_x L^\infty_{T}} \\
  & \lesssim  & 
  \|v\|_{L^2_x L^\infty_{T}} \sum_{j\ge k-3} 2^{(k-j)(s+1)}\|P_{2^j}J^s_x\partial_x v \|_{L^\infty_x L^2_T},
  \end{eqnarray*}
  and by Young's inequality we obtain that 
  $$
  \|P_{HI} \partial_x (v^2)\|_{L^2_T H^s_x}^2 \lesssim \sum_{k\ge 4}  \|P_{2^k} \partial_x (v^2)\|_{L^2_T H^s_x}^2
   \lesssim \|v\|_{L^2_x L^\infty_{T}}^2 \|J^s_x v_x \|_{\widetilde{L^\infty_x L^2_T}}^2.
   $$
   This completes the proof of \eqref{bilinb1}.
  The proof of  \eqref{bilinb2} follows   a similar way and will thus be omitted.
\end{proof}

We are now in position to give the proof of Proposition \ref{apriori w}. 
\begin{proof}[Proof of Proposition \ref{apriori w}] Let $s \ge 1$, $0<T\le 1$, and $\tilde{v}$ and $\tilde{w}$ be extensions of $v$ and $w$ such that $\|\tilde{v}\|_{X^{1-2\theta,\theta}} \le 2\|v\|_{X^{1-2\theta,\theta}_T}$ for all $0 \le \theta \le 1$ and $\|\tilde{w}\|_{X^{1,\frac12,1}} \le 2\|w\|_{X^{1,\frac12,1}_T}$. By the Duhamel principle, the integral formulation associated to \eqref{eqw} writes 
\begin{displaymath} 
\begin{split}
w(t)=\eta(t)w(0)+\eta(t)\int_0^tN(\eta_Te^{iF},\eta_Tv,\eta_TW,\eta_Tw)(\tau)d\tau,
\end{split}
\end{displaymath} 
where $N(e^{iF},v,W,w)$ is defined in $\eqref{eqw}$, for $0<t\le T \le 1$. Therefore, we deduce gathering estimates 
\eqref{prop1.1.2}, \eqref{prop1.2.1}, \eqref{bilin1} and \eqref{bilina1}--\eqref{bilinb2} that 
\begin{displaymath} 
\begin{split}
\|w\|_{X^{s,\frac12,1}}  & \lesssim \|w(0)\|_{H^s}\\ & +p\big(\|w\|_{X^{s,\frac12,1}}, \|v\|_{L^{\infty}_TH^1_x}, \|v\|_{L^4_TW^{1,4}_x},\|v\|_{L^2_xL^{\infty}_T}, 
\|J^s_x v_x \|_{\widetilde{L^\infty_xL^{2}_T}},
\sup_{0\le \theta \le 1}\|v\|_{X^{1-2\theta,\theta}_T} \big),
\end{split}
\end{displaymath}
where $p$ is polynomial at least quadratic in its arguments.  This concludes the proof of estimate \eqref{apriori w.1}, since 
\begin{displaymath} 
\|w(0)\|_{H^s} \lesssim \|J_x^s\big( e^{-iF[v_0]}v_0\big)\|_{L^2} \lesssim \big(1+\|v_0\|_{H^1}^2\big)\|v_0\|_{H^s},
\end{displaymath}
due to Lemma \ref{lemma1}.
\end{proof}

\section{Proof of Theorem \ref{theo1}}
Without loss of generality, we will fix $\epsilon=1$ in this section. First observe that, unlike to the Benjamin-Ono equation, equation \eqref{hoBO} is not invariant under scaling. However, if $v$ is a solution to the equation \eqref{hoBO} on the time interval $[0,T]$ with initial data $v_0$, then for every $0<\lambda < \infty$, $v_{\lambda}(x,t)=\lambda v(\lambda x,\lambda^3 t)$ is a solution to 
\begin{equation} \label{hoBOscaled}
\partial_tv-b\lambda\mathcal{H}\partial^2_xv-
           a \partial_x^3v=c\lambda v\partial_xv-d
           \partial_x(v\mathcal{H}\partial_xv+\mathcal{H}(v\partial_xv)),
\end{equation} 
on the time interval $[0,\lambda^{-3}T]$. Then, since 
\begin{displaymath} 
\|v_{\lambda}(\cdot,0)\|_{H^1} \sim \lambda^{\frac12}\|v_0\|_{L^2}+\lambda^{\frac32}\|\partial_xv_0\|_{L^2} \lesssim 
\lambda^{\frac12}\|v_0\|_{H^1},
\end{displaymath}
we can always force $v_{\lambda}(\cdot,0)$ to belong to $B_{\alpha}$, where $B_{\alpha}$ is the open ball of $H^1(\mathbb R)$ with radius $0<\alpha \ll 1$ and centered at the origin. Therefore the existence and uniqueness of a solution to \eqref{hoBOscaled} on the time interval $[0,1]$ for small initial data in $H^1(\mathbb R)$ will ensure the existence and uniqueness of a solution to \eqref{hoBO} on the time interval $[0,T]$ with $T\sim \lambda^3 \sim \min\{1,\|v_0\|_{H^1}^{-6}\}$ for arbitrary initial data in $H^1(\mathbb R)$. Moreover, using the conservation of the energy $H$ defined in \eqref{H}, which controls the $H^1$-norm,  will imply the global well-posedness of $\eqref{hoBO}$ in $H^1(\mathbb R)$.

Since all the estimates obtained in the precedent sections are still valid for \eqref{hoBOscaled} with implicit constants independent of $0 < \lambda \le 1$ and for sake of simplicity, we will continue working with equation \eqref{hoBO}, in the case $\epsilon=1$, \textit{i.e.}, 
\begin{equation} \label{hoBOepsilon1}
\partial_tv-b\mathcal{H}\partial^2_xv-
           a \partial_x^3v=c v\partial_xv-d
           \partial_x(v\mathcal{H}\partial_xv+\mathcal{H}(v\partial_xv)),
\end{equation}
instead of equation \eqref{hoBOscaled}. The rest of the proof of Theorem \ref{theo1} follows closely our proof for the Benjamin-Ono equation (see Theorem 1.1 in \cite{MP}). For this reason, we will only give a sketch of it. 

As a consequence of the well-posedness theory for more regular solutions obtained in Theorem 1.3 of \cite{LPP}, we have the following result. 
\begin{proposition} \label{smoothsolutions}
For all $v_0 \in H^{\infty}(\mathbb R)$, there exists a positive time $T=T(\|v_0\|_{H^2}) \in (0,1]$ and a solution $v \in C([0,T];H^{\infty}(\mathbb R))$ to equation \eqref{hoBOepsilon1}. Moreover, $T$ is a nondecreasing function of its argument. 
\end{proposition}
Note however that at this point, we still do not know wether those solutions are global or not.

The first step is to obtained \textit{a priori} estimates for those smooth solutions. Let 
$v_0 \in B_{\alpha} \cap H^{\infty}(\mathbb R)$ be given.  Here, we recall that $B_{\alpha}=\{\phi \in H^1(\mathbb R) \ : \ \|\phi\|_{H^1} \le \alpha\}$, where $\alpha>0$ will be choosen sufficiently small. For any $s \ge 1$, we will also denote 
\begin{equation} \label{N} 
N^s_T(v)=\max\big\{\|v\|_{L^{\infty}_TH^s_x}, \|J^s_xv\|_{L^4_{x,T}}, \|v\|_{L^2_xL^{\infty}_T},\|J^s_x\partial_xv\|_{\widetilde{L^{\infty}_xL^2_T}},\|w\|_{X^{s,\frac12,1}_T}, \big\}.
\end{equation}
Then, estimates \eqref{apriori v.1}--\eqref{apriori v.14b} and \eqref{apriori w.1} yield
\begin{equation} \label{theo1.3}
N^s_T(v) \lesssim (1+\|v_0\|_{H^1}^2)\|v_0\|_{H^s}+q(N^1_T(v))N^s_T(v),
\end{equation}
for any $s \ge 1$, where $q$ is a polynomial with no constant term. By continuity, estimate \eqref{theo1.3} for  $s=1$ ensures that there exist two positive constants $\alpha_1$ and $C_1$ such that $N^1_T(v) \le C_1\alpha$, provided $v_0 \in B_{\alpha}$ with $0<\alpha<\alpha_1$. Moreover, using estimate \eqref{theo1.3} again implies that 
\begin{equation} \label{theo1.4} 
\|v\|_{L^{\infty}_TH^s_x} \le N^s_T(v) \lesssim \|v_0\|_{H^s},
\end{equation}
for any $s \ge 1$, provided $\|v_0\|_{H^1} \le \alpha <\alpha_1$. Therefore, by using estimate \eqref{theo1.4} for $s=2$, we can reapply the result of Proposition \ref{smoothsolutions} a finite number of time to extend the solution $v$ to the interval $[0,1]$, as soon as $\|v_0\|_{H^1}$ is small enough. We observe that the scaling argument explained above and the control of the $H^1$-norm by the Hamiltonian $H$ defined in \eqref{H}, allow to extend our solution $v$ globally in time, so that $v \in C(\mathbb R;H^{\infty}(\mathbb R))$.

To prove the uniqueness as well as the continuity of the flow map, we follow our argument in the proof of Theorem 1.1 of \cite{MP} and derive a Lipschitz bound on the flow map $:H^s(\mathbb R) \rightarrow L^{\infty}([0,1];H^s(\mathbb R))$ for initial data having the same low frequency part, and where $L^{\infty}([0,1];H^s(\mathbb R))$ is considered with the norm $N_1^s$ defined in \eqref{N}. The idea is to apply similar estimates to \eqref{apriori v.1}--\eqref{apriori v.14b} and \eqref{apriori w.1} to the difference of two solutions $u:=v_1-v_2$ and to the difference of the gauges $z:=w_1-w_2$. Note that at this point a control on $G:=F[v_1]-F[v_2]$ in $L^{\infty}([0,1]\times \mathbb R)$ is needed and can be obtained exactly as in Lemma 4.1 of \cite{MP} by splitting $G$ between its low and high frequency parts $G:=G_{lo}+G_{hi}$. We use the equation satisfied by $G_{lo}$ and evolving from $0$ to control $G_{lo}$ and Bernstein's inequalities to control $G_{hi}$.

Finally to prove the existence in $H^s(\mathbb R)$, for initial data and the continuity of the flow map, we fix an initial data $v_0 \in B_{\alpha} \cap H^s(\mathbb R)$ and an approximate sequence of initial data $v_0^j=\mathcal{F}_x^{-1}\big(\chi_{|_{[-j,j]}}\mathcal{F}_xv_0\big)$. Then, as explained above, the associated sequence of solutions $\{v^j\}_j$ is a subset of $C([0,1];H^{\infty}(\mathbb R))$. Moreover, since it evolves from initial data having the same low frequency part,  the Lipschitz bound implies that $\{v^j\}_j$ is a Cauchy sequence in all the norms appearing in $N^s_1$ and therefore converges strongly in those norms to a solution $v$ of \eqref{hoBOepsilon1} satisfying $v(\cdot,0)=v_0$ and \eqref{theo1.1}--\eqref{theo1.2}.
\section{Convergence towards the Benjamin-Ono equation in $ H^1(\R) $ when $ \rho=\sqrt{3} \rho_1$}
In this section we prove that \eqref{hoBO} is uniformly in $ \varepsilon>0 $ well-posed in $ H^1(\R) $  whenever the ratio of the density between the two
 fluids is given by  $ \rho=\sqrt{3} \rho_1 $. This will in a  classical way (see for instance \cite{GW}) lead to  Theorem \ref{theo2}. According to Remark \ref{condition}, this condition on this ratio permits to cancel the term $\partial_x P_{+hi}(e^{iF} v^2) $ in \eqref{eqw}.  Note that this term behaves mainly as the nonlinear term of the Benjamin-Ono equation $ \partial_x (v^2)$. Therefore we do not know how deal with this term when $ \varepsilon $ is going to zero. One possibility could be to use the variant of the Bourgain' space introduced in \cite{IKT, KT2} as it was done to study the inviscid limit of the the Benjamin-Ono-Burgers equation in \cite{GuoAll}. However, there is  another problem here since linear and bilinear estimates involving  $ V_\varepsilon $ are not uniform in  $\varepsilon$ (see for instance \eqref{stri1}-\eqref{stri2}).  This is due to the fact that the linear terms $ {\mathcal H} u_{xx} $  and  $\varepsilon \partial_x^3 $  compete together as in the Benjamin equation (cf. \cite{ABR}). This is reflected on the energy $ H $ (see \eqref{H}) by the fact that the $ \dot{H}^1 $ and $ \dot{H}^{1/2}$ components of the quadratic part are of opposite signs.

\subsection{Some linear estimates }
First we establish needed linear estimates on the group $ V_\varepsilon(\cdot) $ (see also 
Lemmas \ref{strichartz}-\ref{smoothing}).
\begin{lemma}
For any $ 0<\varepsilon\le 1$, any $ 0<T\le 1 $ and any $ \varphi \in L^2(\R) $ it holds
\begin{equation}
\| V_\epsilon (t) \varphi \|_{L^4_T L^\infty_x} \lesssim \epsilon^{-1/6} \| \varphi\|_{L^2}\label{stri1}
\end{equation}
and
\begin{equation}
\| V_\epsilon (t)  \varphi \|_{L^6_T L^6_x} \lesssim \epsilon^{-1/9} \| \varphi\|_{L^2}\label{stri2}
\end{equation}
Moreover, denoting by $ P_\epsilon $ a smooth  space Fourier projector on $\{ |\xi| \not\in ]\frac{1}{4\epsilon},\frac{1}{2\epsilon}[ \}$,
 we have
\begin{equation}
\| V_\epsilon (t) P_\epsilon \varphi \|_{L^4_T L^\infty_x} \lesssim  \| \varphi\|_{L^2}\label{stri3}
\end{equation}
and
\begin{equation}
\| V_\epsilon (t) P_\epsilon \varphi \|_{L^6_T L^6_x} \lesssim  \| \varphi\|_{L^2}\label{stri4}
\end{equation}
\end{lemma}
\proof By the $ TT^{*} $ argument it suffices to prove that for $ 0<t\le 1 $,
\begin{equation}
\| V_\epsilon (t) \varphi \|_{L^\infty_x} \lesssim \epsilon^{-1/3}t^{-1/2} \| \varphi\|_{L^2}\label{to1}
\end{equation}
and
\begin{equation}
\| V_\epsilon (t) P_\epsilon  \varphi \|_{L^\infty_x} \lesssim t^{-1/2}\| \varphi\|_{L^2}\label{to2}
\end{equation}
By classical arguments,  \eqref{to1} will be proven if we show
$$
\Bigl\| \int_{\R} e^{i[x\xi+(\xi|\xi|-\epsilon \xi^3)t]}\, d\xi\Bigr\|_{L^\infty_x} \lesssim \epsilon^{-1/3} t^{-1/2}  \; .
$$
Setting $ \theta:=\xi \sqrt{t} $ this is equivalent to prove
\begin{equation}\label{to3}
I_{\epsilon}:=\sup_{|t|\le 1, X\in\R} \Bigl|\int_{\R} e^{i[X\theta+|\theta|\theta-\frac{\epsilon}{\sqrt{t}} \theta^3]}\, d\theta \Bigr| \lesssim \epsilon^{-1/3}
\end{equation}
We set $ \Phi(\theta):=X\theta+\theta^2-\frac{\epsilon}{\sqrt{t}} \theta^3$ and notice that for $ \theta\neq 0 $, 
$$
 \Phi'(\theta):=X+2|\theta|-3 \frac{\epsilon}{\sqrt{t}} \theta^2,\;
 \Phi^{''}(\theta)=2(\text{sgn}\,\theta- 3 \frac{\epsilon}{\sqrt{t}} \theta) \;\mbox{ and  }
 \Phi^{'''} (\theta)= -6  \frac{\epsilon}{\sqrt{t}} \, .
$$
\eqref{to3} is obvious when restricted on $ |\theta|\le 100$. By symmetry we can assume that $ \theta>100 $.
$ X,t $ and $ \epsilon $ being fixed, there exists $ M_{t,\epsilon,X} \ge 100 $ such that
$$
\forall \theta>M, \quad |\Phi'(\theta)| \ge \max( 1+|\theta|, \frac{\epsilon}{\sqrt{t}} \theta^2) \;.
$$
Therefore,  integration by parts yields that  $ I_\epsilon\lesssim 1 $ in this region.\\
Now for $ \theta<M$, we use Van der Corput lemma and that $  |\Phi^{'''} (\theta)|= 6 \frac{\epsilon}{\sqrt{t}} $
 to get $  I_\epsilon\lesssim \epsilon^{-1/3} $. This completes the proof of \eqref{to1}. \\
 To prove \eqref{to2}, we use that $ \xi\not\in ]\frac{1}{4\epsilon},\frac{1}{2\epsilon}[  $ implies
  $ \theta\not\in  ]\frac{\sqrt{t}}{4\epsilon},\frac{\sqrt{t}}{2\epsilon}[  $ and thus
   $ |\Phi^{''}(\theta)|=2|1-3  \frac{\epsilon}{\sqrt{t}} \theta| \gtrsim 1 $.
This yields the result by applying Van der Corput lemma in the region  $\theta<M$.\qed
\begin{lemma}
For any $ 0<\varepsilon\le 1$, any $ 0<T\le 1 $ and any $ \varphi \in L^2(\R) $ it holds
\begin{equation}
\| V_\epsilon (t) \varphi \|_{L^4_x L^\infty_T} \lesssim \epsilon^{-1/4} \| \varphi\|_{L^2}\label{stri5}
\end{equation}
and  for any dyadic integer $ N\ge 10/\varepsilon $, 
\begin{equation}
\| V_\epsilon (t) P_N \varphi \|_{L^\infty_x L^2_t } \lesssim  \varepsilon^{-1/2}  \| P_N \varphi\|_{L^2}\label{stri6}
\end{equation}
\end{lemma}
\proof \eqref{stri5} can be proved by using the change of unknown \eqref{strichartz2} exactly as in Lemma \ref{maximal}  whereas 
 \eqref{stri6} is  proven in  \eqref{smoothing2}.
\subsection{Nonlinear estimates}
According to the linear estimates of Lemmas  \ref{prop1.1}-\ref{prop1.2}, we have to estimate each  terms of  the right-hand side member of \eqref{eqw} in $ X^{1,1/2,1}_\varepsilon $. Recall that the term 
 $ \partial_x P_{+hi} (e^{-iF} v^2) $ cancels in this section  due to the choice of the ratio between $ \rho $ and  $ \rho_1$. 
 \begin{lemma} \label{bilineairepsilon1} Assume that $ \mbox{supp} \,  v \subset \{ |t|\le T\} $ with $ 1<T\le 1$. Then  for any $ 0<\epsilon\le 1 $  it holds
 \begin{equation}
 \Bigl\|\partial_x P_{+hi} (W \partial_x P_{-hi}v)  \Bigr\|_{X^{1,-1/2,1}_\epsilon}\lesssim
  \|w\|_{X^{1,-1/2,1}} (1+\|v\|_{L^\infty_t H^1_x}) \|v\|_{L^\infty_t H^1_x}
 \end{equation}
 \end{lemma}
 \proof
 We rewrite $ \hat{w}$ as $  \eta_{\epsilon} \hat{w} + (1-\eta_{\epsilon}) \hat{w} :=\widehat{w_1}+\widehat{w_2} $ where
  $ \eta_{\epsilon} $ is a smooth even  function with support in $ |\xi|\in  ]\frac{1}{5\epsilon},\frac{5}{9\epsilon}[   $ such that
   $  \eta_{\epsilon}\equiv 1 $ on  $ ]\frac{1}{4\epsilon},\frac{1}{2\epsilon}[$. \\
   By Sobolev embedding we have
   $$
  J:= \Bigl\|\partial_x P_{+hi} (W_2\partial_x P_{-hi}v)  \Bigr\|_{X^{1,-1/2,1}_\epsilon}\lesssim
     T^{1/3} \Bigl\|\partial_x^2 P_{+hi} (W_2\partial_x P_{-hi}v)  \Bigr\|_{L^{2}_T L^2_x}\; .
   $$
   But for any fixed  integer $ k\ge 1 $, making use of the discrete Young's inequality as in the proof of \eqref{apriori v.14b}, we get 
   \begin{eqnarray*}
   \Bigl\|\partial_x^2 P_{2^k} P_{+hi} (W_2\partial_x P_{-hi}v)  \Bigr\|_{L^{2}_t L^2_x} &
   \lesssim & 2^{2k} \sum_{j\ge k } 
      \Bigl\|\partial_x^2 P_{2^k} P_{+hi} (P_{2^j} W_2\partial_x P_{-hi}u)  \Bigr\|_{L^{2}_t L^2_x} \\
      & \lesssim & 2^{2k} \sum_{j\ge k} 2^{-j} \|P_{2^j} w_2 \|_{L^4_T L^\infty_x} 
       \| \partial_x u \|_{L^\infty_t L^2_x} \\
       & \lesssim &  \| \partial_x v \|_{L^\infty_T L^2_x}\sum_{j\ge k}2^{-2(j-k)}
        \|P_{2^j}\partial_x w_2 \|_{L^4_t L^\infty_x}  \\
       & \lesssim & \gamma_k   \|\partial_x w_2 \|_{\widetilde{L^4_t L^\infty_x} }
       \| \partial_x v \|_{L^\infty_t L^2_x} 
   \end{eqnarray*}
   where $ (\gamma_k)\in l^2(\Z_+)$. This is 
   acceptable, according to \eqref{stri3}, since  $ \widehat{w_2} $ cancels on $ ]\frac{1}{4\epsilon},\frac{1}{2\epsilon}[$.\\
   Now to treat the contribution of  $ w_1$   we have to use the resonance relation :
   $$
   \sigma-\sigma_1-\sigma_2=\xi \xi_2(2-3\epsilon \xi_1)
   $$
   Since $\mbox{Supp } \hat{w_1} \subset    ]\frac{1}{5\epsilon},\frac{5}{9\epsilon}[ $, it follows that
   $ \max(|\sigma|,|\sigma_i|)\gtrsim |\xi \xi_2 |$ for this contribution. 
   We set 
   \begin{eqnarray*}
   I&:=&\Bigl\|\partial_x P_{+hi} (W_2\partial_x P_{-hi}v)  \Bigr\|_{X^{1,-1/2,1}_\epsilon}\\
   & \lesssim  & \sum_{N,N_1,N_2}  \Bigl\|P_{N} \partial_x P_{+hi} (P_{N_1} W_2\partial_x P_{-hi} P_{N_2} v)  \Bigr\|_{X^{1,-1/2,1}_\epsilon}:= \sum_{N,N_1,N_2} I_{N,N_1,N_2}
     \end{eqnarray*}   
     Clearly,  
\begin{displaymath}
\begin{split} 
I_{N,N_1,N_2}   =  &\sum_{L} \Bigl[ \int_{\R^2}\frac{\langle\xi \rangle \xi \chi_{\R_+}(\xi) \phi_N (\xi)\psi_L(\xi,\tau)}{\langle \tau-b|\xi|\xi +a \varepsilon \xi^3 \rangle}\Bigl| \\ & \times
\int_{\R^2} 
  \phi_{N_1}(\xi_1) \widehat{W_2}(\xi_1,\tau_1) \xi_2 \phi_{N_2}(\xi_2) \widehat{P_{-hi} v} (\xi_2,\tau_2) d\xi_1\, d\tau_1\Bigr|^2\, d\xi \, d\tau \Bigr]^{1/2}    \; .
 \end{split}
 \end{displaymath}
       We separate the contributions of  different regions with respect to which $ \sigma $ is dominant. Therefore, to calculate  $I_{N, N_1,N_2}$, we split the integration domain $\mathcal{D}$ in the following disjoint regions
\begin{equation} \label{bilincrit4b}
\begin{split}
\mathcal{A}_{N,N_2}&=\big\{(\xi,\xi_1,\tau,\tau_1) \in \R^4 \ | \ 2^{-2} N N_2  \le  |\sigma|\le  2^2 N N_2   \big\}, \\
\mathcal{B}_{N,N_2}&=\big\{(\xi,\xi_1,\tau,\tau_1) \in  \R^4 \ | \ |\sigma_1|\ge \frac 16  N N_2  \ , \ |\sigma|\not\in [2^{-2} N N_2, 2^2 N N_2]  \big\}, \\
\mathcal{C}_{N,N_2}&=\big\{(\xi,\xi_1,\tau,\tau_1) \in  \R^4 \ | \ |\sigma|\not\in [2^{-2} N N_2, 2^2 N N_2]  , \ |\sigma_1|< \frac16 N N_2 \ , |\sigma_2|\ge
 \frac16 NN_2 \big\},
\end{split}
\end{equation}
and denote by $I^{\mathcal{A}_{N,N_2}}_{N, N_1,N_2}$, $I^{\mathcal{B}_{N,N_2}}_{N, N_1,N_2}$, $I^{\mathcal{C}_{N,N_2}}_{N, N_1,N_2}$ the restriction of $I_{N, N_1,N_2}$ to each of these regions. Then, it follows that
$$
I_{N, N_1,N_2}\le I^{\mathcal{A}_{N,N_2}}_{N, N_1,N_2}+I^{\mathcal{B}_{N,N_2}}_{N, N_1,N_2}+I^{\mathcal{C}_{N,N_2}}_{N, N_1,N_2}
$$
 and thus
\begin{equation} \label{bilincrit5}
I^2 \lesssim \sum_{N} \big|  I^{\mathcal{A}}_N \big|^2 +\sum_{N}  \big|  I^{\mathcal{B}}_N\big|^2+\sum_{N}  \big| I^{\mathcal{C}}_N \big|^2,
\end{equation}
where
$$
I^{\mathcal{A}}_N:=\sum_{ N_1,N_2} I^{\mathcal{A}_{N,N_2}}_{N, N_1,N_2}, \; I^{\mathcal{B}}_N:=\sum_{N_1,N_2} I^{\mathcal{B}_{N,N_2}}_{N, N_1,N_2} \ \mbox{and } \ I^{\mathcal{C}}_N:=\sum_{ N_1,N_2} I^{\mathcal{C}_{N,N_2}}_{N, N_1,N_2}\; .
 $$
Therefore,  it suffices to bound $ \sum_{N} \big|  I^{\mathcal{A}}_N \big|^2$, $\sum_{N}  \big|  I^{\mathcal{B}}_N\big|^2$ and $ \sum_{N}  \big| I^{\mathcal{C}}_N \big|^2 $. \\
   $\bullet $ Bound on $\sum_{N}  \big|  I^{\mathcal{B}}_N\big|^2$. We use that $ L^{1+}_{loc} \hookrightarrow {\mathcal B}^{-1/2,1}_{2,loc}  $  to get 
   \begin{eqnarray*}
   I_{2^k} ^{\mathcal{B}}   & \lesssim  &
   T^{1/3}  2^{2k} \sum_{j\ge k } \sum_{q\le j} 
      \Bigl\| P_{2^k} P_{+hi} (P_{2^j} W_2\partial_x P_{-hi}P_{2^q} v)  \Bigr\|_{L^{2}_T L^2_x} \\
      & \lesssim &  T^{1/3}  \sum_{j\ge k}  \sum_{q\le j}  2^{2(k-j)} 2^{-(k+q)/2} \|P_{2^j}\partial_x  w_2 \|_{X^{0,1/2}} 
       \|P_{2^q}  \partial_x v \|_{L^\infty_T L^\infty_x} \\
       & \lesssim &  T^{1/3}   \sum_{j\ge k}  \sum_{q\le j}  2^{\frac{3}{2}(k-j)}  \|P_{2^j} \partial_x w_2 \|_{X^{0,1/2}}
       2^{(q-j)/2}  
       \| P_{2^q}  v\|_{L^\infty_T L^\infty_x}  \\
       & \lesssim &  T^{1/3}  \gamma_k \|w_2\|_{X^{1,1/2}} \|v\|_{L^\infty_T H^1_x}
   \end{eqnarray*}
with $ \|(\gamma_k)\|_{ l^2(\Z_+)}\lesssim 1 $. \\
      $\bullet $  Bound on $ \sum_{N} \big|  I^{\mathcal{A}}_N \big|^2$. Note that since $ |\sigma|\sim N N_2 $ in this region, we do not have to sum on $ L $. 
      \begin{eqnarray*}
   I_{2^k}^{\mathcal{A}}   & \lesssim  &
   \sum_{j\ge k}  \sum_{q\le j}  2^{-2(j-k)} 2^{-(k+q)/2} \|P_{2^j}\partial_x  w_2 \|_{L^2_{tx}} 
       \|P_{2^q}  \partial_x v \|_{L^\infty_T L^\infty_x} \\
       & \lesssim &     \sum_{j\ge k}  \sum_{q\le j}  2^{\frac{3}{2}(k-j)}  \|P_{2^j} \partial_x w_2 \|_{L^2_{tx}}
       2^{(q-j)/2}  
       \| P_{2^q}  v \|_{L^\infty_T L^\infty_x}  \\
       & \lesssim &  T^{1/2}  \gamma_k \|w_2\|_{X^{1,1/2,1}} \|v\|_{L^\infty_T H^1_x}
   \end{eqnarray*}
with $ \|(\gamma_k)\|_{ l^2(\Z_+)}\lesssim 1 $. \\
        $\bullet $  Bound on $\sum_{N}  \big|  I^{\mathcal{C}}_N\big|^2$.
       $$
  \Bigr(      \sum_{N} | I_{N}^{\mathcal{C}}  |^2\Bigl)^{1/2} \lesssim \|{\tilde u}\|_{X^{0,1}} \|{\mathcal F}^{-1}(|\widehat{w_1}|)\|_{L^\infty_{tx}} \lesssim \|w\|_{X^{1,1/2,1}} \|{\tilde v}\|_{X^{0,1}}
       $$
       where $ \widehat{\tilde v}= {\tilde \eta}_\epsilon \hat{v} $ and $ {\tilde \eta}_\epsilon $ is a smooth even function with support in $ \{|\xi|\le \frac{100}{\epsilon}\} $. Here we used that due to the frequency projections together with the frequency localization of $ w_1 $,  the modulus of the frequencies of $ v $ must be less than $ 10/\varepsilon $. 
             Now from the equation \eqref{hoBO} satisfied by $ v $ and the frequency localization of ${\tilde v}$, we deduce that
        $$\|\partial_t V_\epsilon(-t) {\tilde v} \|_{L^2_{tx}} \lesssim\|v v_x \|_{L^2_{tx}} \lesssim \|v\|_{L^\infty_t H^1_x}^2 \; .
       $$
         This yields the desired  result.  \qed
          \begin{lemma}  Assume that $ \mbox{supp} \,  v \subset \{ |t|\le T\} $ with $ 1<T \le 1$. Then  for any $ 0<\epsilon\le 1 $  it holds
          \begin{equation}
\epsilon  \Bigl\|\partial_x P_{+hi} (w\partial_x P_{-hi}v)  \Bigr\|_{X^{1,-1/2,1}_\epsilon}\lesssim
  \|w\|_{X^{1,-1/2,1}} \Bigl[ \epsilon^{1/9} \|v_x\|_{L^4_{tx}}+ (1+\|v\|_{L^\infty_t H^1_x}) \|v\|_{L^\infty_t H^1_x}\Bigr]
 \end{equation}
 \end{lemma}
 \proof
 We rewrite $ \hat{w}$ as $  \zeta_{\epsilon} \hat{w} + (1-\zeta_{\epsilon}) \hat{w} :=\widehat{w_1}+\widehat{w_2} $ where
  $ \zeta_{\epsilon} $ is a smooth even  function with support in $ |\xi|\in  ]\frac{1}{2\epsilon},\frac{1}{\epsilon}[   $ such that
   $  \zeta_{\epsilon}\equiv 1 $ on  $ ]\frac{5}{9\epsilon},\frac{8}{9\epsilon}[$. \\
   By Sobolev embedding and the frequency localization of $w_1$ we have
   $$
      J_1:= \epsilon \Bigl\|\partial_x P_{+hi} (w_1\partial_x P_{-hi}v)  \Bigr\|_{X^{1,-1/2,1}_\epsilon}\lesssim
T^{1/3}    \Bigl\|\partial_x P_{+hi} (w_1\partial_x P_{-hi}v)  \Bigr\|_{L^{2}_t L^2_x} \\
   $$
   We proceed exactly as in the preceding lemma to get 
   \begin{eqnarray*}
   \Bigl\|\partial_x P_{2^k} P_{+hi} (w_1\partial_x P_{-hi}v)  \Bigr\|_{L^{2}_t L^2_x} &
   \lesssim & 2^{2k} \sum_{j\ge k } 
      \Bigl\|\partial_x P_{2^k} P_{+hi} (P_{2^j} w_1\partial_x P_{-hi}v)  \Bigr\|_{L^{2}_t L^2_x} \\
      & \lesssim & 2^{2k} \sum_{j\ge k} 2^{-j} \|P_{2^j} w_1 \|_{L^4_t L^\infty_x} 
       \| \partial_x v \|_{L^\infty_t L^2_x} \\
       & \lesssim & \gamma_k   \|\partial_x w_1 \|_{\widetilde{L^4_t L^\infty_x} }
       \| \partial_x v \|_{L^\infty_t L^2_x} 
   \end{eqnarray*}
   which is acceptable, according to \eqref{to2}, since  $ \widehat{w_1} $ cancels on $ ]\frac{1}{4\epsilon},\frac{1}{2\epsilon}[$.\\
Now for the contribution of $ w_2 $ we will use that thanks to the frequency localization of $ w_2 $ and the resonance relation we have
$$
\max(|\sigma|,|\sigma_i|) \gtrsim\epsilon |\xi\xi_1\xi_2|
$$
on the space-time Fourier support of this contribution. We decompose this contribution as in the preceding lemma but with respect to 
 $  \epsilon N N_1 N_2 $ instead of $ N_1 N_2 $. \\
 $\bullet $ Bound on $\sum_{N}  \big|  I^{\mathcal{B}}_N\big|^2$.
   \begin{eqnarray*}
   I_{2^k} ^{\mathcal{B}}   & \lesssim  &\varepsilon 
    2^{2k} \sum_{j\ge k } \sum_{q\le j} 
      \Bigl\| P_{2^k} P_{+hi} (P_{2^j} w_2\partial_x P_{-hi}P_{2^q} v)  \Bigr\|_{L^{4/3}_T L^2_x} \\
      & \lesssim &    \sqrt{\varepsilon} \sum_{j\ge k}  \sum_{q\le j}  2^{2k} 2^{-j} 2^{-(k+q+j)/2} \|P_{2^j}\partial_x  w_2 \|_{X^{0,1/2}} 
       \|P_{2^q} \partial_x   v \|_{L^4_T L^\infty_x} \\
         & \lesssim &    \sqrt{\varepsilon} \sum_{j\ge k}  \sum_{q\le j}  2^{\frac{3}{2}(k-j)}  \|P_{2^j}\partial_x  w_2 \|_{X^{0,1/2}} 
    \|P_{2^q} D_x^{3/4}  v \|_{L^4_{T,x}} \\
   & \lesssim &   \varepsilon^{1/2-1/9}   \gamma_k \, \|w_x\|_{X^{1,1/2}} (\varepsilon^{1/9} \| v_x \|_{L^4_{T, x}})   \        \end{eqnarray*}
with $ \|\gamma_k\|_{ l^2(\Z_+)}\lesssim 1 $. \\
      $\bullet $  Bound on $\sum_{N} \big|  I^{\mathcal{A}}_N \big|^2$. Note that since $ |\sigma|\sim \varepsilon N N_1 N_2 $ in this region, we do not have to sum on $ L $. 
      \begin{eqnarray*}
   I_{2^k}^{\mathcal{A}}      & \lesssim  &  \sqrt{\varepsilon}\, 
    2^{2k} \sum_{j\ge k } \sum_{q\le j}  2^{-(k+q+j)/2}
      \Bigl\| P_{2^k} P_{+hi} (P_{2^j} w_2\partial_x P_{-hi}P_{2^q} v)  \Bigr\|_{L^2_{tx}} \\
         & \lesssim &    \sqrt{\varepsilon} \sum_{j\ge k}  \sum_{q\le j}  2^{\frac{3}{2}(k-j)}  \|P_{2^j}\partial_x  w_2 \|_{L^4_{tx}} 
    \|P_{2^q} D_x^{1/2}  v \|_{L^4_{Tx}} \\
   & \lesssim &  \varepsilon^{1/2-1/9}   \gamma_k \, \|w_x\|_{X^{1,1/2}} (\varepsilon^{1/9} \| v_x \|_{L^4_{tx}})   \        \end{eqnarray*}
with $ \|\gamma_k\|_{ l^2(\Z_+)}\lesssim 1 $. \\
        $\bullet $  Bound on $\sum_{N}  \big|  I^{\mathcal{C}}_N\big|^2$.
     We separate two subregions. \\
       {\it a)  $\xi_1\le \frac{50}{\epsilon}$} Then we have $ |\xi_2|\le  \frac{50}{\epsilon}$. We write
       $$
        \Bigr(      \sum_{N} | I_{N}^{\mathcal{C}}  |^2\Bigl)^{1/2}\lesssim \ \|{\tilde v}\|_{X^{0,1}} \|{\mathcal F}^{-1}(|\widehat{w_2}|)\|_{L^\infty_{tx}} \lesssim \|w_2\|_{X^{1,1/2,1}} \|{\tilde v}\|_{X^{0,1}}
       $$
     where $ \widehat{\tilde v}= {\tilde \eta}_\epsilon \hat{v} $ and $ {\tilde \eta}_\epsilon $ is a smooth even function with support in $ \{|\xi|\le \frac{100}{\epsilon}\} $
 This is acceptable since, as in the preceding lemma due to the frequency localization of $v$,
        $$
        \|\tilde{v}\|_{X^{0,1}} \lesssim (1+\|v\|_{L^\infty_t H^1_x}) \|v\|_{L^\infty_t H^1_x}\; .
       $$
           {\it b)  $\xi_1> \frac{50}{\epsilon}$}  Then we write
          \begin{eqnarray*}
              I_{2^k}^{\mathcal{C}}  & \lesssim & \sum_{j\ge k} \epsilon \Bigl\|{\mathcal F}^{-1} \Bigl( |\xi^2 \xi_2| |\hat{v} |\, |\widehat{P_{2^j}w_2}|\Bigr)  \Bigr\|_{L^{4/3}_t L^2_x} \\
              & \lesssim &     \| v\|_{X^{-1,1}}  \sum_{j\ge k}    \|{\mathcal F}^{-1}(|\widehat{P_{2^j} \partial_x w_2}|)\|_{L^4_t L^\infty_x}\\
          &  \lesssim  & \gamma_k \| v\|_{X^{-1,1}} \|w\|_{X^{1,1/2}}
          \end{eqnarray*}
           where  we used \eqref{stri3} and the frequency localization of $ w_2$ in this subregion. This is acceptable, since in view of \eqref{hoBO} it is not too hard to check that 
           $$
            \| v\|_{X^{-1,1}}\lesssim (1+\|v\|_{L^\infty_t H^1_x}) \|v\|_{L^\infty_t H^1_x}\; .
            $$
            \qed \vspace*{2mm} \\
            Now the contribution of the term  $ \epsilon\partial_xP_{+hi}(WP_-(vv_x)) $ is easy to estimate as follows
 \begin{eqnarray}
\epsilon     \Bigl\|\partial_xP_{+hi}(WP_-(vv_x)) \Bigr\|_{X^{1,-1/2,1}} &\lesssim
 &  \epsilon\Bigl\|\partial_x^2 P_{+hi}(WP_-(vv_x)) \Bigr\|_{L^2_{x,T}}\nonumber  \\
 & \lesssim & \epsilon^{2/3}  (\epsilon^{1/9}\|w_x\|_{L^6_{x,T}}) (\epsilon^{1/9} \|v\|_{L^6_{x,T}})
  (\epsilon^{1/9} \|v_x\|_{L^6_{x,T}})
    \end{eqnarray}
which is acceptable thanks to \eqref{stri2}. Moreover, the contributions of each term of the left-hand side of \eqref{bilina3}  is controlled by 
$ \|v\|_{L^\infty_{t}H^1_x}^2 $ since, by the frequency projections, only low frequencies of $ v$, $ v^2 $ and $e^{iF}v$ are involved. It thus remains to control the term $ \varepsilon  \|\partial_xP_{+hi}\big(e^{iF}v^3\big)\|_{X^{1,-\frac12,1}} $. This is the aim of the following lemma. 
\begin{lemma}
\begin{eqnarray}
\varepsilon  \|\partial_xP_{+hi}\big(e^{iF}v^3\big)\|_{X^{1,-\frac12,1}}  & \lesssim &  \|v\|_{L^\infty_T H^1}^3 +\varepsilon \|v\|_{L^\infty_t H^1_x}^4 \nonumber\\
& & +  \varepsilon \Bigl[ \|v\|_{L^\infty_T H^1}^3(1+ \|v\|_{L^\infty_T H^1})
+ \| v_{xx} \|_{\widetilde{L^\infty_{x} L^2_t} }\|v\|_{L^4_x L^\infty_{t}} ^2\Bigr]
\end{eqnarray}
\end{lemma} 
\proof. 
First, proceeding as in the proof of Proposition \ref{bilinb},  it is not too hard to see that 
$$
\varepsilon  \|\partial_x P_{+hi}\big(e^{iF}v^3\big)\|_{X^{1,-\frac12,1}}  \lesssim 
 \|v\|_{L^\infty_t H^1_x}^3 + \varepsilon \|v\|_{L^\infty_t H^1_x}^4+ \|v\|_{L^\infty_t H^1_x}^4 +\varepsilon  \| \partial^2_x P_{HI} (v^3) \|_{L^{2}_t L^2_x} \; .
$$
It thus remains to bound the last term of the right-hand side of the above inequality.  We notice  that for $ k\ge 4   $ we may bound 
 its $P_{2^k}$-projection by 
$$
\| \partial_x^2  P_{2^k} v^3 \|_{ L^2_{tx}}\lesssim \Bigl\| \partial_x^2  P_{2^k}\Bigl( \sum_{j\ge k-4 } P_{2^j} v S_{2^{j}}v S_{2^{j}}v  \Bigr)\Bigr\|_{L^2_{tx}} \;.
$$
Hence, proceeding as in Proposition  \ref{bilinb}, it is easy to check that for $ k\ge 4 $, 
$$
\| \partial_x^2  P_{2^k} v^3 \|_{L^{2}_t L^2_x}\lesssim \gamma_k \|v_{xx} \|_{L^\infty_{x} L^2_t }\|v\|_{L^4_x L^\infty_{t}} ^2  $$
where $ \|(\gamma_k)\|_{l^2(\Z_+)} \lesssim 1 $. This completes the proof of the lemma. \qed\vspace{1mm}\\

Finally  to close the estimates we have to control some norms of $ v$ in terms of $ w $. 
\begin{lemma} \label{apriori vbis}
Let $0<T\le1$, $0\le\theta\le 1$, $0<\epsilon \le 1$ and $v$ be a solution to
\eqref{hoBO} in the time interval $[0,T]$. Then, it holds that
\begin{equation} \label{apriori v.33}
\varepsilon^{1/9} \|\partial_x  v\|_{L^6_{x,t}} \lesssim \|v_0\|_{H^1}
+\big(1+\|v\|_{L^{\infty}_TH^1_x}^2\big)\|w\|_{X^{s,1/2,1}}
+ \varepsilon^{1/9} \|v\|_{L^{\infty}_TH^1_x} \|\partial_x  v\|_{L^6_{x,T}},
\end{equation}
\begin{equation} \label{apriori v.34}
\varepsilon^{1/4} \|  v\|_{L^4_{x} L^\infty_t } \lesssim \|v_0\|_{H^1}
+\big(1+\|v\|_{L^{\infty}_TH^1_x}^2\big)\|w\|_{X^{s,1/2,1}} 
+\varepsilon^{1/4}\|v\|_{L^{\infty}_TH^1_x} \|  v\|_{L^4_{x} L^\infty_T} ,
\end{equation}
and 
\begin{equation} \label{apriori v.35} 
\begin{split}
\varepsilon^{1/2} \| \partial_x^2  v\|_{\widetilde{L^{\infty}_x L^2_T}}  \lesssim & \|v_0\|_{H^1}+
\big(1+\|v\|_{L^{\infty}_TH^1_x}^2\big)\|w\|_{X^{s,1/2,1}} \\
  &+\|v\|_{L^{\infty}_TH^1_x} \Bigl( \|v\|_{L^{\infty}_TH^1_x} + \varepsilon^{1/2} \| \partial_x^2 v \|_{\widetilde{L^\infty _xL^{2}_T}}\Bigr).
\end{split}
\end{equation}
\end{lemma}
\proof
 \eqref{apriori v.33} and \eqref{apriori v.34} can be proven exactly as in Proposition \ref{apriori v} with \eqref{stri1}-\eqref{stri5} in hand.
 \eqref{apriori v.35}  can be proven as \eqref{apriori v.14b} in Proposition \ref{apriori v} with \eqref{stri6} in hand, using that by Bernstein inequality, for any $ v\in L^2_{T,x} $ and $ \varepsilon>0 $  it holds 
 $$ \|P_{\le 100/\varepsilon} v\|_{L^\infty_x L^2_T} \lesssim \varepsilon^{-1/2} \|v\|_{L^2_{Tx}} \; .
 $$ 
Gathering Lemmas \ref{bilineairepsilon1}-\ref{apriori vbis} we obtain that \eqref{hoBO} is uniformly well-posed in $ H^1(\R) $ , i.e. 
\begin{proposition} \label{theo3} 
For any $ R>0 $ there exists a positive time $ T=T(R) $ and a positive real number $ C(R) $  such that for all $0<\varepsilon \le 1 $ and any initial data 
 $v_0\in H^1(\mathbb R)$,  with $ \|v_0^j\|_{H^1} \le C(R) $     , it holds 
 \begin{equation}\label{yyy}
\|S_\varepsilon(t) v_0^j \|_{L^\infty_T H^1_x} \le C(R) \; .
 \end{equation}
 Moreover for any couple of initial data $(v_0^1,v_0^2) \in  H^1(\mathbb R)^2  $ with   $ \|v_0^j\|_{H^1} \le C(R) $ and $ P_1(v_0^1)=P_1(v_0^2) $
  it holds 
 \begin{equation}\label{yy} 
  \|S_\varepsilon(t) v_0^1 -S_\varepsilon(t)v_0^2 \|_{L^\infty_T H^1_x} \le C(R) \|v_0^1-v_0^2\|_{H^1} 
 \end{equation}
\end{proposition}
With this proposition in hand, Theorem \ref{theo2} follows by general arguments developed for instance in \cite{GW}. We fix  an initial data  $v_0\in H^1(\R) $ and $ \alpha >0 $ and we would like to prove that for $ \varepsilon>0 $ small enough, 
$$
\|S_\varepsilon(t) v -S(t)v\|_{L^\infty_T H^1_x} \le \alpha 
$$
where $ T=T(\|v_0\|_{H^1}) $. The result for any fixed $ T>0 $ follows by iterating the argument and using the continuity of the flow-map for the  Benjamin-Ono equation. 
First, thanks to \eqref{yy} there exists $ r_\alpha>1 $ such that for all $ \varepsilon\in ]0,1] $ it holds 
$$
\| S_\varepsilon(t) v_0 - S_\varepsilon(t) P_{\le r_\alpha} v_0\|_{L^\infty_T H^1_x} \le \alpha/3\; \hbox{and } 
\| S(t) v_0 - S(t) P_{\le r_\alpha} v_0\|_{L^\infty_T H^1_x} \le \alpha/3 
$$
On the other hand, denoting by $ w_{r_\alpha}$ the gauge transform of $ S(t) P_{r_\alpha} v_0 $ and noticing that 
 the Benjamin-Ono equation \eqref{BO} can be rewritten as  \begin{align*}
& \partial_tv-b\mathcal{H}\partial^2_xv-
           a\epsilon \partial_x^3v-cv\partial_xv+d\epsilon
           \partial_x(v\mathcal{H}\partial_xv-\mathcal{H}(v\partial_xv)) \\
           & =d\epsilon
           \partial_x(v\mathcal{H}\partial_xv-\mathcal{H}(v\partial_xv))  -a\epsilon \partial_x^3v \; ,
\end{align*}
   we can proceed exactly as in the obtention of the Lipschitz bound \eqref{yy} to get  
 $$
 \|S_\varepsilon (t) P_{\le r_\alpha} v_0-S(t) P_{\le r_\alpha} v_0\|_{L^\infty_T H^1_x} \lesssim C(R) T \varepsilon  \Bigl( \| w_{r_\alpha}\|_{L^\infty_T H^3}
  +p\Bigl(N^3_T(S(t)P_{\le r_\alpha} v_0\Bigr)\Bigr)
 $$
 where $ p $ is a polynomial function and the $ N^s_T$-norm is defined in \eqref{N}.
 This yields the result by taking $\varepsilon >0 $ small enough. 
 
 \vspace{0,5cm}

\noindent \textbf{Acknowledgments.} The authors would like to thank Professor Jean-Claude Saut for some helpful comments. D.P. was partially supported by the projects Pronex E-26/110560/2010-APQ1 and FAPERJ E-26/11-564/2008. He also would like to thank the Department of Mathematics at the University of Chicago for the kind hospitality during the redaction of this work.  L.M. is grateful to the Institute Sch\"odinger of Wien for the kind hospitality during the redaction of this work.
 
\bibliographystyle{amsplain}

\end{document}